\documentclass[11pt]{article}
\usepackage{amsmath, amsthm, amssymb, amscd, mathrsfs, amsfonts, bbm, 
graphicx, tikz, tikz-cd, dsfont, pgf, hyperref, mathtools, scalerel, float, 
multicol, enumitem, epsfig, epstopdf, graphicx, relsize, bbm}
\usepackage[margin=2.5cm]{geometry}
\usepackage[utf8]{inputenc}
\usepackage[T1]{fontenc}
\usepackage[font=small, labelfont=bf]{caption}
\usepackage{verbatim}
\usepackage{makecell}
\usepackage{amsmath}
\usepackage{pgfplots}
\pgfplotsset{compat=1.15}
\usetikzlibrary{arrows}
\usepackage{pgfplots}
\newtheorem{theorem}{Theorem}[section]
\newtheorem{corollary}[theorem]{Corollary}
\newtheorem{definition}[theorem]{Definition}

\newtheorem{lemma}[theorem]{Lemma}

\newtheorem{proposition}[theorem]{Proposition}
\newtheorem{remark}[theorem]{Remark}
\newtheorem{corollary-definition}[theorem]{Corollary-Definition}
\numberwithin{equation}{section}
\numberwithin{figure}{section}
\newcommand\diag[4]{%
 \multicolumn{1}{p{#2}|}{\hskip-\tabcolsep
 $\vcenter{\begin{tikzpicture}[baseline=0, anchor=south west, inner sep=#1]
 \path[use as bounding box] (0, 0) rectangle (#2+2\tabcolsep, \baselineskip);
 \node[minimum width={#2+2\tabcolsep-\pgflinewidth}, 
 minimum height=\baselineskip+\extrarowheight-\pgflinewidth] (box) {};
 \draw[line cap=round] (box.north west) -- (box.south east);
 \draw[line cap=round] (box.north west) -- (box.south west);
 \node[anchor=south west] at (box.south west) {#3};
 \node[anchor=north east] at (box.north east) {#4};
 \end{tikzpicture}}$\hskip-\tabcolsep}}

\newcommand{\Bgrouppos}{{\Bgroup}^{+}}

\DeclareMathOperator{\Fl}{Fl}
\DeclareMathOperator{\Gl}{Gl}
\DeclareMathOperator{\UT}{UT}
\DeclareMathOperator{\Bgroup}{B}

\DeclareMathOperator{\Sq}{Sq}
\DeclareMathOperator{\image}{im}
\DeclareMathOperator{\coker}{coker}
\DeclareMathOperator{\Stab}{Stab}
\DeclareMathAlphabet{\altmathcal}{OMS}{cmsy}{m}{n}
\makeatletter \renewcommand{\fnum@figure}{Fig. \thefigure} \makeatother

\input colordvi 

\newcommand{\BComment}[1]{ \color{blue} { #1} \color{black}}
\begin{document}
\title{The mod-2 cohomology groups of low-dimensional unordered flag manifolds and Auerbach bases}
\author{Lorenzo Guerra \and Santanil Jana \and Arun Maiti}
\providecommand{\keywords}[1]{\textit{Keywords:} #1}
\maketitle
\begin{abstract} Unordered flag manifolds are the manifolds of unordered $n$-tuple of mutually orthogonal lines in $\mathbb{R}^n$. In this paper, we develop some basic tools to compute the mod-$2$ cohomology groups of these spaces, and apply them for explicit computation for small $n$. We show that this computation  improves the known estimate of the number of Auerbach bases of normed linear spaces of small dimensions. 
 \end{abstract}
 \keywords{cohomology of groups, unordered flag manifolds, spectral sequence, Auerbach basis, Lusternik-Schnirelmann category}

\section{Introduction} \label{intro}

A complete flag in $ \mathbb{R}^n $ is a maximal sequence of nested linear subspaces
\[
(0 \subsetneq V_1 \subsetneq V_2 \subsetneq \dots \subsetneq V_n = \mathbb{R}^n).
\]
The complete flag manifold over the real numbers $ \Fl_n(\mathbb{R}) $ is a manifold whose points parameterize complete flags in $ \mathbb{R}^n $. This idea has been classically formalized by defining $ \Fl_n(\mathbb{R}) = \Gl_n(\mathbb{R})/B $, the quotient of the general linear group by a Borel subgroup $ B $. By the Gram-Schmidt orthogonalization, the datum of a complete flag is equivalent to that of an $ n $-tuple of pairwise orthogonal lines in $ \mathbb{R}^{n} $. So, $ \Fl_n(\mathbb{R}) \cong O(n)/ T_n$, where $O(n)$ is the orthogonal group and $ T_n $ is the ``torus'' of diagonal matrices with entries $\pm 1$ (a discrete group, not a Lie-theoretic torus). 
The complex counterpart of $ \Fl_n(\mathbb{R}) $, $ \Fl_n(\mathbb{C}) $ is defined by replacing real subspaces of $ \mathbb{R}^n $, $ \Gl_n(\mathbb{R}) $, $ O(n) $, $ T_n$ with complex subspaces of $ \mathbb{C}^n $, 
$ \Gl_n(\mathbb{C}) $, $ U(n) $ and the $ n $-torus $T_n \subseteq U(n) $ consists of diagonal matrices with entries complex numbers of absolute value $1$, respectively. The symmetric group $ \altmathcal{S}_n $ acts on $ \Fl_n(\mathbb{R})$ and $ \Fl_n(\mathbb{C}) $ by permuting the ordered $ n $ lines. We will denote the quotient of this action by $ \overline{\Fl}_n(\mathbb{R}))$ and $ \overline{\Fl}_n(\mathbb{C}))$, respectively, and will refer to them as ``unordered complete flag manifolds''. 

The complete flag manifolds along with other flag manifolds have been extensively studied in the past in algebraic topology and geometry as they play important roles in Lie theory. It is known, for instance, that their cohomology admits a combinatorial description in terms of Schubert polynomials \cite{Fulton-Pragracz} and is isomorphic to the coinvariant algebra of a polynomial algebra with respect to an action of the symmetric group. In contrast, unordered complete flag manifolds have not gained much attention from algebraic topologists. Understanding their topology, besides being interesting in its own right, has important implications for emerging problems in algebraic topology and convex geometry as discussed below. The rational cohomology and homotopy groups of these manifolds are easy to compute (see Proposition \ref{cohoflagcom}, \ref{homotopygroups}). However, since the cohomology groups of the symmetric groups are rich in torsion elements, one expects to gain more useful knowledge by computing their mod-$p$ cohomology groups. In this regard, our modest aim in this paper is to develop tools for the computations of mod-$2$ cohomology groups of these spaces in lower orders and present explicit descriptions in dimensions $3, 4, 5$ in the real case and in dimension $3$ in the complex case. 

We embed the hyperoctahedral group $\Bgroup_n$ as a subgroup of $O(n)$ via its reflection representation and we let $\Bgrouppos_n$ to be the subgroup $\Bgroup_n$ consisting of elements with positive determinant. This is known in the literature as its alternating subgroup \cite{Brenti-Reiner-Roichman}.
 We will show \ref{normalizer} that there are fiber sequences
\begin{gather*}
O(n) \hookrightarrow \overline{\Fl}_n(\mathbb{R}) \twoheadrightarrow B\Bgroup_n \tag*{and} \\
SO(n) \hookrightarrow \overline{\Fl}_n(\mathbb{R}) \twoheadrightarrow B\Bgrouppos_n.
\end{gather*}
For the last fiber sequence, in contrast with the first one, the fundamental group of the base acts trivially on the fiber. Thus, there is a cohomological Serre spectral sequence
\[
E_2^{*,*} = H^*(SO(n); \mathbb{F}_2) \otimes H^*(B\Bgrouppos_n; \mathbb{F}_2) \Rightarrow H^*(\overline{\Fl}_n(\mathbb{R}); \mathbb{F}_2)
\]
that do not involve complicated local coefficients. This is our main computational tool.

We feed the spectral sequence above explicit descriptions of the cohomology of $SO(n)$ and $\Bgrouppos_n$. The cohomology of $SO(n)$ is well known \cite{Brown}. In \S \ref{sec:chbpos}, we describe a method to derive the cohomology of $\Bgrouppos_n$ from a combinatorial description of the cohomology of $\Bgroup_n$ recently produced by L. Guerra in \cite{LG21}.

We then use a combination of geometric methods and the algebraic structure of the spectral sequence to compute its differentials. More specifically, the differentials on the second pages of the spectral sequences are determined by geometric reasoning (see Lemma \ref{geometric machinery d2}), and the higher pages algebraically. Our geometric approach is based on the configuration space model for $ B\Bgrouppos_n $ described in \S \ref{sec:recollection}. This allows us to describe the geometric cycles generating the groups $ H_2(B\Bgrouppos_n) $ explicitly. We then represent a basis of the cohomology groups via Thom classes of submanifolds transversal to the cycles, and this enables us to compute the pairing between $ H_2(B\Bgrouppos_n) $ and $ H^2(B\Bgrouppos_n) $ by counting the (transverse) intersections. In our algebraic arguments, we mostly exploit the multiplicative structure of the spectral sequence, the Transgression Theorem, and its connection with Steenrod squares. Finally, we compute the differential $ d_2 $ by dualizing and switching to homology.

In the complex case, we compute the cohomology of the unordered flag manifold of order $3$ in \S \ref{sec:complex}. We use the Serre spectral sequence associated to the fiber sequence 
\[ \Fl_3(\mathbb{C}) \to \overline{\Fl}_3(\mathbb{C}) \to B\altmathcal{S}_3. \] 
The cohomology of $\Fl_3(\mathbb{C})$ and $B\altmathcal{S}_3$ are known \cite{fulton1996, Giusti-Salvatore-Sinha}. The differentials in the spectral sequence can then be determined by standard algebraic methods, with the technical complication of local coefficients.
\\
In the follow-up work \cite{Guerra-Santanil}, based on some of the results (e.g., cohomology of $\Bgrouppos_n$) and techniques developed here, with a more delicate involvement of the $ E_\infty $-structure on $ BU $ and $ BO $, we generalize the results of this work to encompass complete unordered flag manifolds of all dimensions. Further work for deeper understanding of the topology of these spaces (ring structure and integer cohomology) will be taken up in future articles.
\\
In \S \ref{auerbach}, we discuss some direct implications of the computational results mentioned above in certain problems of Banach geometry. Let $\mathbb{X}$ be a $n$ dimensional (real or complex) Banach space and $S_{\mathbb{X}}$ denote its unit sphere. A basis $\altmathcal{B}=\{v_1, \cdots, v_n\}$ of $\mathbb{X}$ is called an \textit{Auerbach basis} if $v_i \in S_{\mathbb{X}}$ and there is a basis ${v^1, ..., v^n}$ of the dual space $X^{*}$ satisfying
\begin{equation*}
 v^i(v_j) = \delta_{ij}, \quad \mathrm{and} \ v^i \in S_{\mathbb{X^*}}\ \mathrm{for} \ i, j = 1, 2, \cdots, n. 
\end{equation*}
\\
In \cite{WW17}, Weber--Wojciechowski provided an estimate of the number of Auerbach bases of a finite-dimensional Banach space using topological methods. Here, one identifies bases that differ only by permutation or multiplication by scalars of absolute value one. In other words, two bases are said to be equivalent if they lie in the same orbit of the action of $\Bgroup_n$ on $GL_n(\mathbb{R})$ in the real case and of the action of $N_{U(n)}(T_n)$ on $U(n)$ in the complex case.
The method of their proof, in the real case, can be summed up in two main steps, more details of which are provided in \S \ref{auerbach}. In the first step, the bifurcation points of the $|\det|$ restricted to $S_{\mathbb{X}, \neq 0}^n $, 
\[
S_{\mathbb{X}, \neq 0}^n:= \{ (v_1, v_2, \cdots, v_n) \in S_{\mathbb{X}}^n \ | \ \det (v_1, v_2, \cdots, v_n) \neq  0\}, 
\]
can be shown to correspond to Auerbach bases. A proof of this fact is given in \cite{WW17} in the real case. Here, both real and complex case follows (see Corollary \ref{detandab}) from the characterization of the Auerbach bases we provide in Theorem \ref{int2auerbach}. Moreover, distinct bifurcation points of $|\det|$ restricted to   $S_{\mathbb{X}, \neq 0}^n/\Bgroup_n$ in the complex case) correspond to distinct equivalence classes of Auerbach bases. 
For $C^{1}$-differentiable norms, Auerbach bases can be interpreted as the critical points of the determinant restricted to $S_{\mathbb{X}, \neq 0}^n $ \cite{MS23}. 
\newline
In the second step, using Lusternick-Schnirelmann theory, it is shown that the number of bifurcation points of the $|\det|$ on $S_{\mathbb{X}, \neq 0}^n/\Bgroup_n$ is bounded below by 
$ \text{cat} (S_{\mathbb{X}, \neq 0}^n/ \Bgroup_n)$. Here, the category of a space $X$ is denoted by $\text{cat} (X)$. For generic $C^2$-differentiable norm, using Morse theory, the number of critical points up to equivalence is bounded below by $\text{rank} (H^*(S_{\mathbb{X}, \neq 0}^n/ \Bgroup_n))$.
 \newline 
Since, $  S_{\mathbb{X}, \neq 0}^n/ \Bgroup_n \cong \overline{\Fl}_n (\mathbb{R}) \times D^{n(n-1)/2}$ (see Proposition \ref{transition}), where $D^{n}$ is the open unit ball in $\mathbb{R}^n$, 
 \begin{equation} \label{catequality}
 \text{cat}(S_{\mathbb{X}, \neq 0}^n/\Bgroup_n) = \text{cat}(\overline{\Fl}_n (\mathbb{R})) 
, \quad \text{rank}(H^*(S_{\mathbb{X}, \neq 0}^n/\Bgroup_n)) = \text{rank} (H^*(\overline{\Fl}_n (\mathbb{R}))).
 \end{equation}
The method of proof in the complex case is very similar, with the group $\Bgroup_n$ replaced by $N_{U(n)}(T_n)$. In \cite{WW17}, the estimates of the category and rank were based on the known results about the cohomology of $SO(n), U(n)$, and the flag manifolds. While the estimate of $\text{cat}(\overline{\Fl}_n (\mathbb{R}))$ turned out to be optimal, the other quantities were not. In \S \ref{auerbach}, we show that our computational results described above can improve the estimates for these invariants. 

The unordered flag manifolds also appear in commutative K-theory. It has been noted (cf. \cite{AGV20}) that the classifying space for commutativity of $SO(n)$ (for $U(n)$ in the complex case), $B_{com}(SO(n))$ ($B_{com}(U(n))$ respectively) is a homotopy colimit of certain spaces that include unordered flag manifolds. 

\textbf{Convention.} All (co)homology groups, throughout this paper, are with $\mathbb{F}_2$ coefficients unless explicitly mentioned otherwise.

\section{Preliminaries}

This section introduces some definitions and notations and recalls some known results.

\subsection{Unordered flag manifolds}

\begin{definition}
The complete flag manifold $\Fl_n(\mathbb{F})$ in $\mathbb{F}^n$ is a space that consists of sequences of vector subspaces:
\[ (0 = V_0 \subsetneq V_1 \subsetneq \dots \subsetneq V_n = \mathbb{F}^n) \]
where $\dim(V_k) = \dim(V_{k-1}) + 1$ for all $k$. It is topologized as a subspace of the product of Grassmannian manifolds $\prod_{l=1}^n \text{Gr}_l(\mathbb{F}^n)$.
When $\mathbb{F} = \mathbb{R}$ or $\mathbb{C}$, we refer to them as the real or complex complete flag manifold, respectively.
\end{definition}

We consider the standard flag $ F_0 = (0 \subsetneq \mathbb{R}\{e_1\} \subsetneq \mathbb{R}\{e_1, e_2\} \subsetneq \dots \subsetneq \mathbb{R}\{e_1, \dots, e_n\} = \mathbb{R}^n ) $, where $ \{ e_1, \dots, e_n \} $ is the standard basis of $ \mathbb{R}^n $. The group $ \Gl_n(\mathbb{R}) $ acts transitively on complete flags, and the stabilizer of $ F_0 $ is the Borel subgroup $ \UT_n $ of upper triangular matrices.
This identifies $ \Fl_n(\mathbb{R}) $ with the quotient $ \Gl_n(\mathbb{R})/\UT_n $.

Since by the Gram-Schmidt process, any matrix can be written as a product of an upper triangular and an orthogonal matrix, the subgroups $ O(n) $ and $ \UT_n $ generate the whole group $ \Gl_n(\mathbb{R}) $.
Moreover, the intersection of $ \UT_n $ with $ O(n) $ is the ``torus'' $ T_n \subseteq O(n) $ consisting of all diagonal matrices with $ \pm1 $ entries on the diagonal. Therefore, $ \Fl_n(\mathbb{R}) $ is isomorphic to $ O(n)/T_n $.

We consider the Coxeter group of Type $ \Bgroup_n $. As defined classically (see for instance \cite{Humphreys}), it is the subgroup of $ \Gl_n(\mathbb{R}) $ generated by the reflections
\begin{gather*}
(x_1, \dots, x_n) \mapsto (-x_1, x_2, \dots, x_n) \tag*{and}\\
(x_1, \dots, x_n) \mapsto (x_1, \dots, x_{i-1}, x_{i+1}, x_i, x_{i+2}, \dots, x_n) ;\quad \forall 1 \leq i < n.
\end{gather*}

Note that $ \Bgroup_n $ is the subgroup of $ O(n) $ generated by $ T_n $ and the subgroup of permutation matrices $ \Sigma_n $.

Let $ \pi \colon \Bgroup_n \to \mathbb{Z}/2\mathbb{Z} $ be the group homomorphism defined by mapping each simple reflection to $ 1 + 2\mathbb{Z} $, then it is easy to see that $ \Bgrouppos_n = \ker(\pi) $. This alternative definition is introduced also in \cite{Brenti-Reiner-Roichman}. Further, by studying the stabilizer of a complete flag, it is easy to show that the normalizer $ N_{O(n)}(T_n) $ of the Torus in $ O(n) $ is $ \Bgroup_n $. Hence
\[
N_{SO(n)}(T_n \cap SO(n)) = N_{O(n)}(T_n) \cap SO(n) = \Bgrouppos_n.
\]
\begin{definition}
The real complete unordered flag manifold is the homogeneous space
\[
\overline{\Fl}_n(\mathbb{R}) = \frac{O(n)}{\Bgroup_n}.
\]
\end{definition}
$ \overline{\Fl}_n(\mathbb{R}) $ can be understood as a quotient of $ \Fl_n(\mathbb{R}) $ more intuitively. A complete flag
\[
\{ 0 \} = V_0 \subsetneq V_1 \subsetneq \dots \subsetneq V_n = \mathbb{R}^n
\]
is uniquely determined by the $ n $-tuple of mutually orthogonal lines $ (\ell_1,\dots,\ell_n) $, where $ \ell_i $ is the orthogonal complement of $ V_{i-1} $ inside $ V_i $. $ \overline{\Fl}_n(\mathbb{R}) $ is the quotient of $ \Fl_n(\mathbb{R}) $ by the action of the symmetric group $ \Sigma_n $ that permutes the order of these $ n $ lines.

Similarly, we define complex complete unordered flag manifolds.
In $ U(n) $ there is a subgroup $ T_n $ consisting of diagonal matrices. This is a Lie-theoretic torus and is isomorphic to $ U(1)^n $. The normalizer $ N_{U(n)}(T_n) $ is the subgroup generated by $ T_n $ and the subgroup $ \Sigma_n $ of permutation matrices. This is identified with the wreath product $ \Sigma_n \wr U(1) $.
\begin{definition}
The complex complete unordered flag manifold is the homogeneous space
\[
\overline{\Fl}_n(\mathbb{C}) = \frac{U(n)}{\Sigma_n \wr U(1)}.
\]
\end{definition}
There is an interpretation of $ \overline{\Fl}_n(\mathbb{C}) $ in terms of sets of pairwise orthogonal complex lines in $ \mathbb{C}^n $, analogous to the real case.

\subsection{Serre spectral sequence}

As $ \overline{\Fl}_n(\mathbb{R}) = SO(n)/\Bgrouppos_n $, there is a fiber sequence
\[
SO(n) \to \overline{\Fl}_n(\mathbb{R}) \to B\Bgrouppos_n,
\]
where $ B\Bgrouppos_n $ is the classifying space of $ \Bgrouppos_n $. The action of $ \Bgrouppos_n $ on the fiber $ SO(n) $ is given by left multiplication. This action is homotopy trivial because $SO(n)$ is path connected, so any path from an element of $ \Bgrouppos_n $ to the identity yields a trivializing homotopy.

We summarize our remarks in the following statement.
\begin{proposition} \label{normalizer}
Let $ n \in \mathbb{N} $. In the homotopy fiber sequence
\[
SO(n) \to \overline{Fl}_n(\mathbb{R})  \to B\Bgrouppos_n
\]
the action of $ \Bgrouppos_n = \pi_1(B\Bgrouppos_n) $ on the homology and cohomology of the fiber is trivial.

Therefore, there is an induced spectral sequence (see \cite{JM01})
\[
E_2^{p,q} = H^p(B\Bgrouppos_n) \otimes H^q(SO(n)) \Rightarrow H^{p+q}(\overline{\Fl}_n(\mathbb{R})).
\]
\end{proposition}

The cohomology groups of $ SO(n) $ are well known and stated below for future reference. 
\begin{theorem}[\cite{Brown,AT}] \label{chso3} 
\[H^*(SO(n), \mathbb{F}_2) \cong \bigoplus_{i \ \mathrm{odd}} \mathbb{F}_2[\beta_i]/(\beta^{p_i}_i)\]
\vspace{-0.2cm}
where $|\beta_i|= i$ and $p_i$ is the smallest power of $2$ such that $|\beta_i^{p_i}| \geq n$.
\end{theorem}
The additive structure of the cohomology of $ B\Bgrouppos_n $ will be computed later on.
Another possible approach is to use the Leray-Cartan spectral sequence associated to the spectral sequence 
\begin{align*}
    \altmathcal{S}_n \to \Fl_n(\mathbb{R}) \to \overline{\Fl}_n(\mathbb{R}) \tag*{and} \\
    \altmathcal{S}_n \to \Fl_n(\mathbb{C}) \to \overline{\Fl}_n(\mathbb{C}).
\end{align*} 
The cohomology of $\altmathcal{S}_n$ (\cite{Giusti-Salvatore-Sinha}), $\Fl_n(\mathbb{R})$ and $\Fl_n(\mathbb{C})$ are known. We state the latter for future reference. 
\begin{theorem} \cite{fulton1996} \label{cohoflagreal} 
The cohomology ring $H^* (\Fl_n(\mathbb{R}))$ is given by
\[H^* ({\Fl}_n(\mathbb{R}))= H^*( BT_n)/(\sigma_1, \sigma_2, \cdots, \sigma_n)= \mathbb{F}_2[x_1, x_2, \cdots, x_n]/(\sigma_1, \sigma_2, \cdots, \sigma_n), \]
where $x_i$'s corresponds to generators in degree one and $\sigma_i(x_1, x_2, \cdots, x_n)$ are elementary symmmetric functions.
\end{theorem}

\begin{theorem} (\cite{fulton1996}) \label{cohoflagcom}
The cohomology ring $H^*({\Fl}_n(\mathbb{C}); \mathbb{Z})$ is generated by elements in degree $ 2 $ 
\[H^*({\Fl}_n(\mathbb{C}); \mathbb{Z})= H^*( BT_n, \mathbb{Z})/(\sigma_1, \sigma_2, \cdots, \sigma_n)= \mathbb{Z}[x_1, x_2, \cdots, x_n]/(\sigma_1, \sigma_2, \cdots, \sigma_n) . \]
 
\end{theorem}
 
The rational cohomology of both real and complex unordered flag manifolds can be easily derived using standard methods.
\begin{proposition} \label{rationalcoho} 
\[ H^* (\overline{\Fl}_n(\mathbb{R}); \mathbb{Q})= H^* (SO(n); \mathbb{Q})   \quad \mathrm{and} \quad    H^*(\overline{\Fl}_n(\mathbb{C}); \mathbb{Q}) =  \mathbb{Q}. \]
\end{proposition}
\begin{proof}
As it is classically known, since $ \Bgrouppos_n $ is discrete, the Serre spectral sequence associated to the fiber sequence of Proposition \ref{normalizer} collapses at the second page and provides an isomorphism with the subalgebra of invariants of the cohomology of the fiber:
\[ [H^*(SO(n)/\Bgrouppos_n ; \mathbb{Q})) \cong [H^*(SO(n); \mathbb{Q}))]^{\Bgrouppos_n} \]  
Since the action of $ \Bgrouppos_n $ on the fiber is homotopically trivial,
 \[ [H^*(SO(n); \mathbb{Q})]^{\Bgrouppos_n}  \cong H^*(SO(n); \mathbb{Q}) \]
Thus the real case of the statement now follows. 

In the complex case, we use the Serre spectral sequence associated to the fiber sequence
\[  \Fl_n(\mathbb{C})  \to  \overline{\Fl}_n(\mathbb{C}) \to  B \altmathcal{S}_n\] 
which yields an isomorphism
\[ H^*(\overline{\Fl}_n(\mathbb{C}); \mathbb{Q}) \cong [H^q(\overline{\Fl}_n(\mathbb{C}); \mathbb{Q})]^{\altmathcal{S}_n} \cong \mathbb{Q} \] 
concentrated in degree $ 0 $ by Theorem \ref{cohoflagcom}.
\end{proof}

The homotopy groups of the unordered flag manifolds can be easily derived from the known results about the homotopy groups of $O(n)$ and $U(n)$ \cite{AT}. 
\begin{proposition}  The homotopy groups of the unordered flag manifolds are given by the following.  \label{homotopygroups}
 \[ \pi_i(\overline{\Fl}_n(\mathbb{R}))= \pi_i(O(n)) \quad \mathrm{for} \ i \geq 2, \quad  \ \ \pi_1 (\overline{\Fl}_n(\mathbb{R}))= \Bgroup_n.\]

 \[ \pi_i(\overline{\Fl}_n(\mathbb{C}))= \pi_i(U(n)) \quad \mathrm{for} \ i > 2, \quad \pi_2(\overline{\Fl}_n(\mathbb{C}))=\mathbb{Z}^{n-1}, \ \ \pi_1 (\overline{\Fl}_n(\mathbb{C}))= \altmathcal{S}_n.\]
\end{proposition}
\begin{proof} 

The proof follows using the long exact sequence of homotopy groups associated to the fibrations 

\[\Bgroup_n \hookrightarrow O(n) \longrightarrow {} \overline{\Fl}_n(\mathbb{R}) \quad \text{and} \quad N(n) \hookrightarrow U(n) \longrightarrow {} \overline{\Fl}_n(\mathbb{C}).\]
 We leave the details to the reader.
\end{proof}

\section{The cohomology of \texorpdfstring{$ \Bgroup_n $ }{text} and \texorpdfstring{$ \Bgrouppos_n $ }{text}} \label{sec:recollection}

The cohomology of $ \Bgroup_n $ was computed by Guerra \cite{LG21}.  The direct sum of these cohomology groups exhibits a rich structure called the Hopf ring.
\begin{definition} A graded Hopf ring over a field $ \mathbb{F} $ is a ring object in the category of co-commutative $ \mathbb{F} $-coalgebras.
 Explicitly, it consists of a septuple $(V, \Delta, \odot, \cdot, S, \eta, \varepsilon)$ such that:
 \begin{itemize}
 \item $ (A, \odot, \Delta, S, \eta, \varepsilon) $ is a graded bi-commutative Hopf algebra, with product $ \odot \colon A \otimes A \to A $, coproduct $ \Delta \colon A \to A \otimes A $, antipode $ S \colon A \to A $, unit $ \eta \colon \mathbb{F} \to A $ and counit $ \varepsilon \colon A \to \mathbb{F} $, 
 \item $ (A, \cdot, \Delta, \varepsilon) $ is a graded bialgebra, 
 \item and the following distributivity law (expressed in Sweedler notation) holds for all $ x, y, z \in A $
\[ x \cdot (y \odot z)=\sum_{(x)} \pm (x_{(1)} \cdot y) \odot (x_{(2)} \cdot z).\]
 \end{itemize}
\end{definition}
As we work mod $ 2 $, the sign in the distributivity law is not important.

\subsection{Hopf ring structure on the cohomology of \texorpdfstring{$\Bgroup_n $}{text}} 
The results presented in this subsection are stated either explicitly or implicitly in \cite{LG21}.

There is a configuration-space-like model for the classifying space of $ \Bgroup_n $ studied by Salvetti--De Concini \cite{Salvetti:00}. For a reflection $\tau \in \Bgroup_n$, let $H_{\tau} \subset \mathbb{R}^n$ denote the hyperplane consisting of its fixed points.
Let $ \mathbb{R}^\infty $ be the direct limit of $ \mathbb{R}^n $ for $ n \in \mathbb{N} $.
Then we define
\[
E\Bgroup_n = (\mathbb{R}^\infty \otimes \mathbb{R}^n) \setminus \bigcup_{\tau \in \Bgroup_n} (\mathbb{R}^\infty \otimes H_\tau).\]
The obvious action of $ O(n) $ on $ \mathbb{R}^\infty \otimes \mathbb{R}^n $ restricts to an action of $ \Bgroup_n $ on $ E\Bgroup_n $. Since $ E\Bgroup_n $ is contractible, its quotient $ B\Bgroup_n = E\Bgroup_n/\Bgroup_n $ is a model for the classifying space of $ \Bgroup_n $. 

The image of the composition
\[ \Bgroup_n \times \Bgroup_m \to O(n) \times O(m) \stackrel{\oplus}{\to} O(n+m)
\]
is contained in $ \Bgroup_{n+m} $. 
Therefore, $ \Bgroup_n \times \Bgroup_m $ acts on $ E\Bgroup_{n+m} $.
Using the configuration models explained above, there are a $ \Bgroup_n \times \Bgroup_m $-equivariant projection map $ p \colon E\Bgroup_{n+m} \to E\Bgroup_n \times E\Bgroup_m $ that splits every sequence of $ (n+m) $ points of $ \mathbb{R}^\infty $ into two sequences, one with $ n $ points and the other with $ m $ points. It induces an homotopy equivalence $$ \frac{E\Bgroup_{n+m}}{\Bgroup_n \times \Bgroup_m} \simeq B\Bgroup_n \times B\Bgroup_m $$ after passing to quotients. This combines with a covering map to provide the topological product
\[
\mu \colon B\Bgroup_n \times B\Bgroup_m \simeq \frac{E\Bgroup_{n+m}}{\Bgroup_n \times \Bgroup_m} \to \frac{E\Bgroup_{n+m}}{\Bgroup_{n+m}} = B\Bgroup_{n+m}.
\]

We fix a basis for $ \mathbb{R}^\infty $. There are two families of cohomology classes $ \alpha_k \in H^k(B\Bgroup_k) $ and $ \gamma_{k, l} \in H^{l(2^k-1)}(B\Bgroup_{l2^k}) $ as follows:
\begin{itemize}
\item $ \alpha_k $ is the Thom class of the finite codimensional submanifold $ X_{\alpha_k} \subseteq B\Bgroup_n $ arising from $ n $-tuples of points $ (p_1, \dots, p_n) $
whose first coordinates in the given basis are all $ 0 $.
\item $ \gamma_{k, l} $ is the Thom class of the finite codimensional properly immersed manifold $ X_{\gamma_{k, l}} \looparrowright B\Bgroup_n $ arising from $ l2^k $-tuples of points that can be subdivided into $ l $ groups of $ 2^k $ points each, where points in each group share their first coordinate.
\end{itemize}
As always in this section, we refer to \cite{LG21} for the precise definitions. Let $ A_B = \bigoplus_n H^*(B\Bgroup_n) $. We consider the following structural morphisms in $ A_B $:
\begin{enumerate}
\item the projection $ \varepsilon \colon A_B \to \mathbb{F} $ onto the $ 0^{th} $ component $ H^*(B\Bgroup_0) = \mathbb{F} $, 
\item the injection of the $ 0^{th} $ component $ \eta \colon \mathbb{F} \to A_B $, 
\item the coproduct $ \Delta \colon A_B \to A_B \otimes A_B $ determined, component-by-component, by the cohomological map $ \mu^* \colon H^*(B\Bgroup_{n+m}) \to H^*(B\Bgroup_n) \otimes H^*(B\Bgroup_m) $ associated to the geometric product $ \mu \colon B\Bgroup_n \times B\Bgroup_m \to B\Bgroup_{n+m} $, 
\item the product $ \odot \colon A_B \otimes A_B \to A_B $, determined, component-by-component, by the cohomological transfer map $ \mu^{!} $ (that goes in the opposite direction with respect to $ \mu^* $) associated to the geometric product $ \mu $, 
\item and the usual component-wise cup product $ \cdot \colon A_B \otimes A_B \to A_B $.
\end{enumerate}
\begin{theorem}[\cite{LG21}] \label{cohob}
$ (A_B, \cdot , \odot, \Delta, \eta, \varepsilon)$ is a commutative Hopf ring generated by the classes $\gamma_{k, l}$ ($k \geq 0, l\geq 1$) and $\alpha_n$ ($n>1)$ satisfying the following relations:
\begin{enumerate}
 \item the $\cdot $ product of elements belonging to different components is zero, 
 \item $\gamma_{k, l} \odot \gamma_{k, m}= \binom{l+m}{k} \gamma_{k, l+m}$, 
 \item $\alpha_m \odot \alpha_n =\binom{l+m}{m} \alpha_{m+n}$, 
 \item $\Delta(\gamma_{k, l})= \sum^l_{i=0}\gamma_{k, i} \otimes \gamma_{k, l-i}$, 
 \item $\Delta(\alpha_n)= \sum^n_{i=0}\alpha_i \otimes \alpha_{n-i}$.
\end{enumerate}
\end{theorem}

\begin{definition}[\cite{LG21}] \label{basis B}
A gathered block in $ A_B $ is a formal cup product of Hopf ring generators $ \gamma_{k, l} $ and $ \alpha_k $ all belonging to the same component. We define the index of $ \gamma_{k, l} $ and $ \alpha_l $ to be $ k $ and $ 0 $ respectively. The profile of a gathered block is the multiset of the indexes of generators that appear there. A Hopf monomial is a transfer product of gathered blocks with pairwise different profiles.
\end{definition}

The set of Hopf monomials forms a basis for $ A_B $ as a $ \mathbb{F} $-vector space. The Hopf ring relations above allow us to algorithmically compute the cup product of any number of Hopf monomials.

We remark that the classes $ \gamma_{k,l} $ arise primarily from the cohomology of the symmetric groups $ \altmathcal{S}_n $ (see Definition 4.7 and Theorem 4.9 of \cite{Giusti-Salvatore-Sinha}), while the kernels of the restriction map $ H^*(B\Bgroup_n) \to H^*(B\altmathcal{S}_n) $ induced by the standard inclusions $\altmathcal{S}_n \to \Bgroup_n $ for $ n \in \mathbb{N} $ give the Hopf ring ideal generated by the classes $ \alpha_n $.

\subsection{Geometric cochains of \texorpdfstring{$ \Bgroup_n $}{text}} 
In \cite{FMS22}, the authors present a framework for understanding cohomology classes geometrically. Specifically, we require information about the behavior of Thom classes under pullbacks, including the following facts:
\begin{itemize}
\item If $ f \colon M \to N $ is a smooth map between manifolds and $ X \subseteq N $ is a proper submanifold transverse to $ f $, then the image of the Thom class of $ X $ with respect to the pullback $ f^* \colon H^*(N) \to H^*(M) $ is the Thom class of $ f^{-1}(X) $.
\item if $ f \colon M \to N $ is a smooth covering between manifolds and $ X \subseteq M $ is a proper submanifold on which the restriction of $ f $ is injective, then the image of the Thom class of $ X $ with respect to the transfer map $ f^{!} \colon H^*(M) \to H^*(N) $ is the Thom class of $ f(X) $.
\end{itemize}
The authors of \cite{FMS22} study these maps in the context of finite-dimensional manifolds. Our spaces $ E\Bgroup_n $ are infinite-dimensional, but we use the fact that $ E\Bgroup_n $ is the direct limit of the sequence of spaces $ \{ E_{k, n} := E\Bgroup_n \cap ( \mathbb{R}^k \otimes \mathbb{R}^n) \}_{k \in \mathbb{N}} $ and that the inclusion maps $ E_{k, n} \to E_{k+1, n} $ are highly connected for large $ k $ to reduce to computations in the finite-dimensional approximants $ E_{k, n} $.
The cohomology of $ E\Bgroup_n $ is the inverse limit of the cohomologies of $ E_{k, n} $, and thus, by restricting to $ E_{k, n} $ for $ k $ large enough, one can deduce that the two statements above still hold for finite codimensional submanifold $ X \subseteq E\Bgroup_n $ or $ X \subseteq B\Bgroup_n $.

These properties allow us to describe many classes in $ A_B $ as Thom classes of certain submanifolds $B\Bgroup_n$. For example, $ \gamma_{1, 1} \odot \alpha_1 $ is the Thom class of the submanifold in $ B\Bgroup_3 $ consisting of three points in $ \mathbb{R}^\infty $ that can be divided into two groups, one with two points that share their first coordinate, and the other with a single point with first coordinate equal to $ 0 $.

In what follows, we will often identify some cohomology classes in $ H^*(B\Bgroup_n) $ as Thom classes by using the properties of the two facts stated in this subsection, without further mention.

\subsection{The cohomology of \texorpdfstring{$ \Bgrouppos_n$}{Bn+}} \label{sec:chbpos}

We first describe a procedure to compute the cohomology of the normalizer of the torus in $SO(n)$, $ H^*(B\Bgrouppos_n)$ for all $n$, and then list them explicitly for small $n$. In the following statements, we let $ \rho_n \colon H^*(B\Bgroup_n) \to H^*(B\Bgrouppos_n) $ to  be the cohomological restriction map and $ \tau_n \colon H^*(B\Bgrouppos_n) \to H^*(B\Bgroup_n) $ be the associated transfer map.

Since the double covering
\begin{equation} \label{fibrentob} S^0 \hookrightarrow B\Bgrouppos_n \rightarrow B\Bgroup_n
\end{equation}
has fiber $ S^0 $, their cohomology groups fit into a Gysin long exact sequence:
\begin{equation} \label{Gysin}
\cdots \rightarrow H^{k-1}(B\Bgroup_n) \xrightarrow{\partial} H^k(B\Bgroup_n) \xrightarrow{\rho_n} H^k(B\Bgrouppos_n) \xrightarrow{\tau_n}H^k(B\Bgroup_n) \xrightarrow{\partial}H^{k+1}(B\Bgroup_n) \rightarrow \cdots
\end{equation}
So, knowledge of the connecting homomorphisms $ \partial $ and the groups $ H^*(B\Bgroup_n) $ fully determine the groups $ H^*(B\Bgrouppos_n) $.

\begin{lemma} \label{connectinghom}
$ \partial $ is equal to the cup product with the class $ e = \gamma_{1, 1} \odot 1_{n-2} + \alpha_1 \odot 1_{n-1} $.
\end{lemma}
\begin{proof}
It is known that the connecting homomorphism is the cup product with the Euler class $ e \in H^1(B\Bgroup_n) $. By theorem \ref{cohob}, 
\[ e = \lambda_{1, n}(\gamma_{1, 1} \odot 1_{n-2}) + \lambda_{2, n}(\alpha_1 \odot 1_{n-1}) \] for some $\lambda_{1, n}, \lambda_{2, n} \in \mathbb{F}_2$. We only need to prove that $ \lambda_{1, n} = \lambda_{2, n} = 1 $.

In the following diagram of bundles involving \ref{fibrentob} and the inclusion map $\altmathcal{S}_n \to \Bgroup_n $ the right square is a pullback.
\begin{center}
\begin{tikzcd}
S^0 \arrow[hook]{r} \arrow[equal]{d} & B A_n \arrow[two heads]{r} \arrow{d} & B\altmathcal{S}_n \arrow{d} \\
S^0 \arrow[hook]{r} & B\Bgrouppos_n \arrow[two heads]{r} & B\Bgroup_n
\end{tikzcd}
\end{center}
The Euler class of the top bundle is non-zero by \cite[Lemma 3.9]{Giusti-Sinha} and, as remarked in \S \ref{sec:recollection}, 
$ \alpha_1 \odot 1_{n-1} $ belongs to the kernel of the restriction map from $ H^*(B\Bgroup_n) $ to $ H^*(B\altmathcal{S}_n) $. It follows from the functoriality of the Euler class that $ e $ maps to a non-zero class under this map. Thus, $ \lambda_{1, n}=1 $ for all $n \geq 1$.

In order to determine $\lambda_{2, n} $, we see that, for $n=1$, $ \Bgroup_1 \cong \mathbb{F}_2 $ is a cyclic group of order $2$ and \ref{fibrentob} is its classifying bundle, 
whose Euler class is non-zero, hence equal to $ \alpha_1 $.
So $\lambda_{2, 1}= 1$.
For $ n > 1 $, the pullback of \ref{fibrentob} along the inclusion $ i \colon \Bgroup_1 \to \Bgroup_n $ is isomorphic to it, 
so $ e $ restricts to $ \alpha_1 $ in the cohomology of $ \Bgroup_1 $.
Since
\[i^*(\lambda_{1, n}(\gamma_{1, 1} \odot 1_{n-2}) + \lambda_{2, n}(\alpha_1 \odot 1_{n-1})) = \lambda_{2, n}(\alpha_1) \]
we deduce that $\lambda_{2, 1}=1$ for all $n$. 
\end{proof}
Since the cohomology of $ SO(n)/\Bgrouppos_n $ for degrees above $ n(n-1)/2 $ is zero, for our computation in spectral sequence, it suffices to know the cohomology of $ \Bgrouppos_n $ up to that degree. 



In the remaining part of this paper, by an abuse of notation, we will use the same notation for a class $ x \in H^*(B\Bgroup_n) $ and its image $ \rho_n(x) \in H^*(B\Bgrouppos_n) $ when there is no confusion regarding the group we are working with.
\begin{proposition} \label{cohomology Bpos}
There is a non-zero element $ t \in \coker(\rho_4) $ in degree $ 3 $. Bases for $ H^*(B\Bgrouppos_n) $ in low degrees are given in the following table.
\begin{table}[H] 
\caption{Bases for $ (\image(\rho_n))^d \subseteq H^d(B\Bgrouppos_n) $ depending on $ n $ and $ d $.}
\begin{center}
\begin{tabular}{|c|p{0.2\textwidth}p{0.3\textwidth}p{0.45\textwidth}|}
\hline
\diag{.1em}{.5cm}{$ \mathbf{d} $}{$\mathbf{n}$} & $ \mathbf{3} $ & $ \mathbf{4} $ & $ \mathbf{5} $ \\
\hline
{$ \mathbf{0} $} & $ \{ 1_3 \} $ & $ \{ 1_4 \} $ & $ \{ 1_5 \} $ \\
{ $ \mathbf{1} $} & $ \{ \gamma_{1, 1} \odot 1_1 \} $ & $ \{ \gamma_{1, 1} \odot 1_2 \} $ & $ \{ \gamma_{1, 1} \odot 1_3 \} $ \\
$ \mathbf{2} $ & $ \{ \gamma_{1, 1}^2 \odot 1_1, \alpha_2 \odot 1_2 \} $ & $ \{ \gamma_{1, 1}^2 \odot 1_2, \gamma_{1, 2}, \alpha_2 \odot 1_2 \} $ & $ \{ \gamma_{1, 1}^2 \odot 1_3, \gamma_{1, 2} \odot 1_1, \alpha_2 \odot 1_3 \} $ \\
$ \mathbf{3} $ & $ \{ \gamma_{1, 1}^3, \gamma_{1, 1}^2 \odot \alpha_1, \alpha_3 \} $ & $ \{ \gamma_{1, 1}^3 \odot 1_2, \gamma_{1, 1} \odot \alpha_1^2 \odot 1_1, \gamma_{1, 1} \odot \alpha_2, \gamma_{2, 1}, \alpha_3 \odot 1_1, t \} $ & $ \{ \gamma_{1, 1}^3 \odot 1_3, \gamma_{1, 1}^2 \odot \gamma_{1, 1} \odot 1_1, (\gamma_{1, 1} \alpha_2) \odot 1_3, \alpha_3 \odot 1_2 \} $ \\
$ \mathbf{4} $ & $ \dots $ & $ \dots $ & $ \{ \gamma_{1, 1}^4 \odot 1_3, \gamma_{1, 1}^3 \odot \gamma_{1, 1} \odot 1_1, \gamma_{1, 2}^2 \odot 1_1, (\gamma_{1, 1}^2 \alpha_2) \odot 1_3, \gamma_{1, 1}^2 \odot \alpha_2 \odot 1_1, \gamma_{1, 1} \odot \alpha_3, \alpha_2^2 \odot 1_3 \} $ \\
$ \mathbf{5} $ & $ \dots $ & $ \dots $ & $ \{ \gamma_{1, 1}^5 \odot 1_3, \gamma_{1, 1}^4 \odot \gamma_{1, 1} \odot 1_1, \gamma_{1, 1}^3 \odot \gamma_{1, 1}^2 \odot 1_1, (\gamma_{1, 1}^3 \alpha_2) \odot 1_3, \gamma_{1, 1}^3 \odot \alpha_2 \odot 1_1, \gamma_{1, 1}^2 \odot \alpha_3, (\gamma_{1, 1} \alpha_2^2) \odot 1_3, (\gamma_{1, 1} \alpha_2) \odot \alpha_2 \odot 1_1, \gamma_{1, 1} \odot \alpha_2^2 \odot 1_1, \alpha_2^2 \odot \alpha_1 \odot 1_1, \alpha_5 \} $ \\
\hline
\end{tabular}
\end{center}
\label{basis table}
\end{table}
\end{proposition}
\begin{proof}
To prove this proposition, we observe that the Gysin sequence \ref{Gysin} splits into the short exact sequences
\[
0 \to \frac{H^d(B\Bgroup_n)}{(e)} \to H^d(B\Bgrouppos_n) \to \ker(e \cdot \_ \colon H^d(B\Bgroup_n) \to H^{d+1}(B\Bgroup_n)) \to 0.
\]
Therefore, a basis for $ H^d(B\Bgrouppos_n) $ is obtained by taking the union of the bases for $ H^d(B\Bgroup_n) $ and $ \ker(e \cdot \_ \colon H^d(B\Bgroup_n) \to H^{d+1}(B\Bgroup_n)) $. 
The product structure in $ H^*(B\Bgroup_n) $ is complicated but can be made explicit from the Hopf ring presentation. Hence, by a straightforward but rather long calculation, we can determine the multiplication by $ e $ on the Hopf monomial basis \ref{basis B} and thus determine the kernel and cokernel of this map. We defer these computations to Appendix \ref{appendix}.
\end{proof}

\section{Cohomology of \texorpdfstring{$ \overline{\Fl}_n(\mathbb{R}) $}{real complete flag manifolds}} \label{sec:sss}
In this section, we first prepare the ground for the general spectral sequence argument that we will use to compute $ H^*(\overline{\Fl}_n(\mathbb{R}) $.
Our starting points are the cohomology groups of $ B\Bgrouppos_n $, which we have calculated in low degrees in the previous section, and the cohomology groups of $SO(n)$, which we have already\ stated in Theorem \ref{chso3}. The following geometric construction is crucial in determining the differentials in spectral sequence. 

\begin{lemma} \label{geometric machinery d2}
Consider the fibration $ SO(n) \hookrightarrow SO(n)_{h\Bgrouppos_n} \twoheadrightarrow B\Bgrouppos_n $.
The following statements are true:
\begin{enumerate}
\item If $ \varGamma \colon [0, 1]^2 \to B\Bgrouppos_n $ is a mod $ 2 $ geometric cycle such that $ \forall 0 \leq t \leq 1: \varGamma(0, t) = \varGamma(1, 1-t) $, $ \varGamma(t, 1) = \varGamma(0, 1-t) $, then the image of any path in $ [0, 1]^2 $ from $ (0, 0) $ to $ (1, 1) $ via $ \varGamma $ represents an involution $ y \in \Bgrouppos_n = \pi_1(B\Bgrouppos_n) $. 
\item Let $ \varGamma $ be as in $ 1 $. If $ y $ has eigenvalue $ - 1 $ with multiplicity $ 2 $, then the homology class of $ \varGamma $ in $ H_2(B\Bgrouppos_n) $ transgresses to the fundamental class of the submanifold $ SO(2) $ in $ H_1(SO(n)) $.
\item Assume that $ \overline{\varGamma} \colon [0, 1]^2 \to B\Bgroup_n $ is a geometric chain satisfying the following hypotheses:
\begin{itemize}
\item $ \overline{\varGamma}(0, 0) = \overline{\varGamma}(0, 1) = \overline{\varGamma}(1, 0) = \overline{\varGamma}(1, 1) $;
\item $ \overline{\varGamma}(0, t) = \overline{\varGamma}(1, t) $ and $ \overline{\varGamma}(t, 0) = \overline{\varGamma}(t, 1) $ for all $ 0 \leq t \leq 1 $, and $ \overline{\varGamma}|_{[0, 1] \times \{0\}} $ represents an involution in $ \pi_1(B\Bgroup_n) $;
\item The two loops $ t \mapsto \overline{\varGamma}(t, 0) $ and $ t \mapsto \overline{\varGamma}(1, t) $ are conjugate in $ \Bgroup_n = \pi_1(B\Bgroup_n) $ via an element $ x \in \Bgroup_n $.
\end{itemize}
Then there is a mod $ 2 $ geometric cycle $ \varGamma \colon [0, 1]^2 \to B\Bgrouppos_n $ whose projection onto $ B\Bgroup_n $ is homologous to $ \overline{\varGamma} $ mod $ 2 $ and satisfying the condition in $ 1 $, with $ y $ equal to the loop obtained by restricting $ \overline{\varGamma} $ to the main diagonal of the square.
\end{enumerate}
\end{lemma}
\begin{proof}
\begin{enumerate}
\item Choose a path $ \ell \colon [0, 1] \to [0, 1]^2 $ from $ (0, 0) $ to $ (1, 1) $. The hypotheses guarantees that $ \varGamma(0, 0) = \varGamma(1, 1) $, thus $ \varGamma \ell $ is a loop. Let $ y \in \pi_1(B\Bgrouppos_n) $ be its class. Since all such paths are homotopic in $ [0, 1]^2 $, $ y $ does not depend on $ \ell $.
By our hypotheses the left-top path in $ \partial([0, 1]^2) $ yields the inverse of the bottom-right path in $ \partial([0, 1]^2) $, hence this independence implies the equation $ y^{-1} = y $. It follows that $ y^2 = 1 $.
\item
To compute the transgression of $ [\varGamma] \in H_2(B\Bgrouppos_n) $, we first construct an explicit lifting $ \phi \colon [0, 1]^2 \to SO(n)_{h\Bgrouppos_n} $ of $ \varGamma $.

First, lift $ \varGamma $ to a continuous map $ \tilde{\varGamma} \colon [0, 1]^2 \to E\Bgrouppos_n $. Choose $ \ast = \tilde{\varGamma}(0, 0) $ as a basepoint for $ E\Bgrouppos_n $. Restricting 
$\tilde{\varGamma}$ to the bottom-right path in $ \partial([0, 1]^2) $ and to the diagonal of the square define two paths $ \ell' $ and $ \ell'' $, respectively. By our hypotheses, the left-top path is $ (y\ell')^{-1} $, where $ y $ acts on paths in $ E\Bgrouppos_n $ via the monodromy action.
Let $ \varGamma^+ $ and $ \varGamma^- $ be the restriction of $ \tilde{\varGamma} $ to the two triangles in which the square is cut by the diagonal (see Figure~\ref{fig:fig1}).
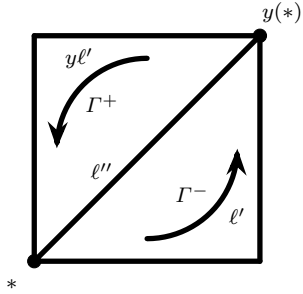
\begin{figure}[H]
    \centering
    \begin{tikzpicture}[line cap=round, line join=round, >=triangle 45, x=3cm, y=3cm]
\clip(-0.5, -0.5) rectangle (1.5, 1.5);
\draw [line width=1pt] (0, 0)-- (1, 0);
\draw [line width=1pt] (1, 0)-- (1, 1);
\draw [line width=1pt] (1, 1)-- (0, 1);
\draw [line width=1pt] (0, 1)-- (0, 0);
\draw [line width=1pt] (0, 0)-- (1, 1);
\draw [line width=1pt, -{Stealth[length=10pt]}] plot[domain=0:1.57, variable=\t] ({0.5 + 0.4*sin(\t r)}, {0.5 - 0.4*cos(\t r)});
\draw [line width=1pt, -{Stealth[length=10pt]}] plot[domain=0:1.57, variable=\t] ({0.5 - 0.4*sin(\t r)}, {0.5 + 0.4*cos(\t r)});
\begin{scriptsize}
\draw [fill=black] (0, 0) circle (1.5pt);
\draw [fill=black] (1, 1) circle (1.5pt);
\draw [color=black] (-0.05, -0.05) node {$\ast$};
\draw [color=black] (1.1, 1.1) node {$ y(\ast) $};
\draw [color=black] (0.9, 0.2) node {$\ell'$};
\draw [color=black] (0.2, 0.9) node {$y\ell'$};
\draw [color=black] (0.3, 0.4) node {$\ell'' $};
\draw [color=black] (0.7, 0.3) node {$\varGamma^-$};
\draw [color=black] (0.3, 0.7) node {$\varGamma^+$};
\end{scriptsize}
\end{tikzpicture}
    \caption{Illustrations of the paths $\ell'$, $\ell''$, $y\ell'$, and the triangles $\varGamma^+$ and $\varGamma^-$.}
    \label{fig:fig1}
\end{figure}

Let $ \cdot $ denote the concatenation of loops. Choose a path $ \ell''' $ from $ \ast $ to $ y(\ast) $ in $ E\Bgrouppos_n $. The paths $ (y\ell''') \cdot \tilde{\varGamma}\ell' $ and $ (y\tilde{\varGamma}\ell'') \cdot \ell''' $ are homotopic relative to endpoints. Therefore, there exists a homotopy $ H_1 $ from $ \tilde{\varGamma}\ell' $ and $ y\tilde{\varGamma}\ell'' $ whose restriction to the two endpoints is $ \ell''' $ and $ y \ell''' $, respectively.
Similarly, since $ y ^2 = 1 $ in $ \Bgrouppos_n $, there is a homotopy $ H_2 $ between $ (y\ell''') \cdot \ell''' $ and the constant path relative to endpoints.
We glue together copies of $ H_1 $, $ H_2 $, $ \varGamma^+ $ and $ \varGamma^- $ as depicted in Figure~\ref{fig:efig2} to obtain a continuous map $ \phi_1 \colon [0, 1]^2 \to E\Bgrouppos_n $.
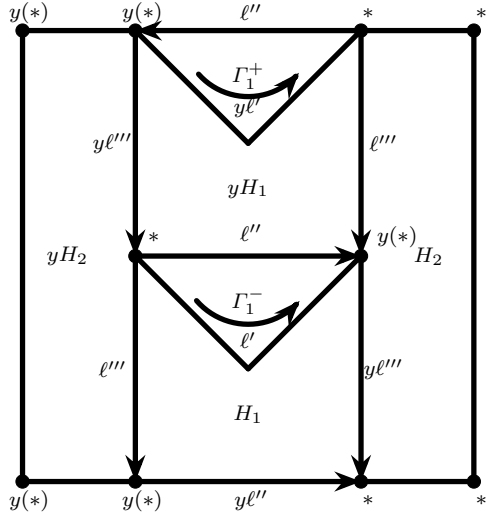
\begin{figure}[H]
    \centering
    \begin{tikzpicture}[line cap=round, line join=round, >=triangle 45, x=3cm, y=3cm]
\clip(-0.5, -0.5) rectangle (2.5, 2.5);
\draw [line width=1pt] (0, 2)-- (0, 0);
\draw [line width=1pt] (1.5, 0)-- (2, 0);
\draw [line width=1pt] (0, 0)-- (0.5, 0);
\draw [line width=1pt, -{Stealth[length=10pt]}] (0.5, 0)-- (1.5, 0);
\draw [line width=1pt] (2, 0)-- (2, 2);
\draw [line width=1pt] (2, 2)-- (1.5, 2);
\draw [line width=1pt] (0, 2)-- (0.5, 2);
\draw [line width=1pt, -{Stealth[length=10pt]}] (1.5, 2)-- (0.5, 2);
\draw [line width=1pt, -{Stealth[length=10pt]}] (0.5, 2)-- (0.5, 1);
\draw [line width=1pt, -{Stealth[length=10pt]}] (0.5, 1)-- (0.5, 0);
\draw [line width=1pt, -{Stealth[length=10pt]}] (1.5, 2)-- (1.5, 1);
\draw [line width=1pt, -{Stealth[length=10pt]}] (1.5, 1)-- (1.5, 0);
\draw [line width=1pt, -{Stealth[length=10pt]}] (0.5, 1)-- (1.5, 1);
\draw [line width=1pt] (0.5, 2)-- (1, 1.5);
\draw [line width=1pt] (1, 1.5)-- (1.5, 2);
\draw [line width=1pt] (0.5, 1)-- (1, 0.5);
\draw [line width=1pt] (1, 0.5)-- (1.5, 1);
\draw [shift={(1, 2)}, line width=1pt, -{Stealth[length=10pt]}] plot[domain=3.862563977561689:5.591263924940729, variable=\t]({1*0.29287843515318424*cos(\t r)+0*0.29287843515318424*sin(\t r)}, {0*0.29287843515318424*cos(\t r)+1*0.29287843515318424*sin(\t r)});
\draw [shift={(1, 1)}, line width=1pt, -{Stealth[length=10pt]}] plot[domain=3.854202968383322:5.568778752608255, variable=\t]({1*0.30249987246196647*cos(\t r)+0*0.30249987246196647*sin(\t r)}, {0*0.30249987246196647*cos(\t r)+1*0.30249987246196647*sin(\t r)});
\begin{scriptsize}
\draw [fill=black] (0, 0) circle (1.5pt);
\draw[color=black] (0.030370370370371394, -0.0841975308641966) node {$y(\ast)$};
\draw [fill=black] (2, 0) circle (1.5pt);
\draw[color=black] (2.0333333333333354, -0.0841975308641966) node {$\ast$};
\draw [fill=black] (2, 2) circle (1.5pt);
\draw[color=black] (2.0333333333333354, 2.0792592592592576) node {$\ast$};
\draw [fill=black] (0, 2) circle (1.5pt);
\draw[color=black] (0.030370370370371394, 2.0792592592592576) node {$y(\ast)$};
\draw [fill=black] (0.5, 2) circle (1.5pt);
\draw[color=black] (0.5320987654321001, 2.0792592592592576) node {$y(\ast)$};
\draw [fill=black] (0.5, 0) circle (1.5pt);
\draw[color=black] (0.5320987654321001, -0.0841975308641966) node {$y(\ast)$};
\draw [fill=black] (1.5, 0) circle (1.5pt);
\draw[color=black] (1.5316049382716066, -0.0841975308641966) node {$\ast$};
\draw [fill=black] (1.5, 2) circle (1.5pt);
\draw[color=black] (1.5316049382716066, 2.0792592592592576) node {$\ast$};
\draw[color=black] (1.0140740740740755, -0.081456790123456879) node {$y\ell''$};
\draw[color=black] (1.0140740740740755, 2.0808641975308626) node {$\ell''$};
\draw [fill=black] (0.5, 1) circle (1.5pt);
\draw[color=black] (0.5820987654321001, 1.0837037037037023) node {$\ast$};
\draw [fill=black] (1.5, 1) circle (1.5pt);
\draw[color=black] (1.6616049382716066, 1.0837037037037023) node {$y(\ast)$};
\draw[color=black] (1.0140740740740755, 1.104938271604937) node {$\ell''$};
\draw[color=black] (0.9982716049382732, 1.6604938271604923) node {$y\ell'$};
\draw[color=black] (1, 0.6276543209876531) node {$\ell'$};
\draw[color=black] (0.2, 1) node {$yH_2$};
\draw[color=black] (1.8, 1) node {$H_2$};
\draw[color=black] (1, 1.3) node {$yH_1$};
\draw[color=black] (1, 0.3) node {$H_1$};
\draw[color=black] (1, 0.8) node {$\varGamma_1^-$};
\draw[color=black] (1, 1.8) node {$\varGamma_1^+$};
\draw[color=black] (0.4, 1.5) node {$y\ell'''$};
\draw[color=black] (0.4, 0.5) node {$\ell'''$};
\draw[color=black] (1.6, 1.5) node {$\ell'''$};
\draw[color=black] (1.6, 0.5) node {$y\ell'''$};
\end{scriptsize}
\end{tikzpicture}
    \caption{Gluing of $H_1$, $H_2$, $\varGamma^+$, and $\varGamma^-$.}
    \label{fig:efig2}
\end{figure}

Second, we choose a path $ \ell'''' $ from $ 1_{SO(n)} $ to $ y $ in $ SO(n) $ and we let $ \phi_2 \colon [0, 1]^2 \to SO(n) $ be a map that behaves on $ \partial ([0, 1]^2) $ as depicted in Figure~\ref{fig:fig3}. Note that such a map exists because the loop described on the boundary of the square is nullhomotopic in $ SO(n) $.
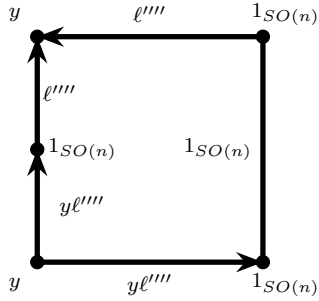
\begin{figure}[H]
    \centering
    \begin{tikzpicture}[line cap=round, line join=round, >=triangle 45, x=3cm, y=3cm]
\clip(-0.5, -0.5) rectangle (1.5, 1.5);
\draw [line width=1pt, -{Stealth[length=10pt]}] (0, 0)-- (1, 0);
\draw [line width=1pt] (1, 0)-- (1, 1);
\draw [line width=1pt, -{Stealth[length=10pt]}] (0, 0)-- (0, 0.5);
\draw [line width=1pt, -{Stealth[length=10pt]}] (0, 0.5)-- (0, 1);
\draw [line width=1pt, -{Stealth[length=10pt]}] (1, 1)-- (0, 1);
\begin{scriptsize}
\draw [fill=black] (0, 0) circle (1.5pt);
\draw [fill=black] (1, 0) circle (1.5pt);
\draw [fill=black] (1, 1) circle (1.5pt);
\draw [fill=black] (0, 1) circle (1.5pt);
\draw [fill=black] (0, 0.5) circle (1.5pt);
\draw [color=black] (-0.1, -0.1) node {$ y $};
\draw [color=black] (-0.1, 1.1) node {$ y $};
\draw [color=black] (1.1, -0.1) node {$1_{SO(n)} $};
\draw [color=black] (1.1, 1.1) node {$1_{SO(n)}$};
\draw [color=black] (0.2, 0.5) node {$1_{SO(n)} $};
\draw [color=black] (0.5, 1.1) node {$\ell''''$};
\draw [color=black] (0.5, -0.1) node {$y \ell'''' $};
\draw [color=black] (0.1, 0.75) node {$\ell''''$};
\draw [color=black] (0.2, 0.25) node {$y\ell''''$};
\draw [color=black] (0.8, 0.5) node {$1_{SO(n)}$};
\end{scriptsize}
\end{tikzpicture}
    \caption{Depiction of the map $\phi_2$.}
    \label{fig:fig3}
\end{figure}

The projection of $ \phi = \phi_1 \times \phi_2 $ onto $ (E\Bgrouppos_n \times SO(n))/\Bgrouppos_n $ onto $ B\Bgrouppos_n $ is homologous to $ \varGamma $ mod $ 2 $ because the two copies of $ H_3 $ and $ H_4 $ project to twice the same chain. Moreover, we note that the top and bottom edges of the square have the same image under $ \phi_1 $, but with opposite orientations.
by gluing them we obtain a geometric cocycle supported on a M\"obius strip, whose boundary lies in a fiber $ SO(n) $.

We deduce that $ [\varGamma] $ transgresses to $ [\partial(\phi)] = [\ell''''] + [y\ell''''] $. If the involution $ y $ has eigenvalue $ - 1 $ with multiplicity $ 2 $, then it is conjugate to the image $ \tau $ of $ -1 \in SO(2) $ by the standard inclusion $ SO(2) \hookrightarrow SO(n) $. Since $ SO(n) $ is path connected, every inner automorphism of $ SO(n) $ acts trivially on $ H_*(SO(n)) $. We can therefore assume without loss of generality that $ y = \tau $, for which we can choose $ \ell'''' $ to be a path in $ SO(2) $ and we deduce that $ [\ell''''] + [\tau \ell''''] = [SO(2)] $.

\item To simplify the notation, let $ a = [t \in [0, 1] \mapsto \overline{\varGamma}(t, 0)] $ and $ b = [t \in [0, 1] \mapsto \overline{\varGamma}(1, t)] $ in $ \pi_1(B\Bgroup_n) $.
Let $ \tilde{\varGamma} \colon [0, 1]^2 \to E\Bgroup_n $ be a lifting of $ \overline{\varGamma} $, and choose $ \ast = \tilde{\varGamma}(0, 0) $ as a basepoint for $ E\Bgroup_n $.
Fix a path $ \ell_x $ from $ \ast $ to $ x(\ast) $. Let $ \ell_a $ and $ \ell_b $ be $ \tilde{\varGamma}|_{[0, 1] \times \{0\}} $ and $ \tilde{\varGamma}|_{\{0\} \times [0, 1]} $, respectively. By hypothesis $ b = xax^{-1} $, so there must be a homotopy $ H_3 $ between the paths $ \ell_b $ and $ x\ell_a $ whose restriction to the two endpoint are $ \ell_x $ and $ b \ell_x $. Moreover, since $ a^2 = 1 $, there is a homotopy $ H_4 $ between the paths $ \ell_a $ and $ (a \ell_a)^{-1} $ relative to their endpoints. We combine $ \tilde{\varGamma} $ and two copies of $ H_3 $ and $ H_4 $ as depicted in Figure~\ref{fig:fig4} to obtain a geometric chain $ \tilde{\varGamma}^e $ in $ E\Bgroup_n $.

\begin{figure}[H]
    \centering
    \begin{tikzpicture}[line cap=round, line join=round, >=triangle 45, x=1.6cm, y=1.6cm]
\clip(-2.5, -1) rectangle (4, 3);
\draw [line width=1pt, -{Stealth[length=10pt]}] (0, 2)-- (0, 0);
\draw [line width=1pt, -{Stealth[length=10pt]}] (0, 0)-- (2, 0);
\draw [line width=1pt, -{Stealth[length=10pt]}] (2, 0)-- (2, 2);
\draw [line width=1pt, -{Stealth[length=10pt]}] (0, 2)-- (2, 2);
\draw [line width=1pt, -{Stealth[length=10pt]}] (0, 0)-- (-1, 0);
\draw [line width=1pt, -{Stealth[length=10pt]}] (0, 2)-- (-1, 2);
\draw [line width=1pt, -{Stealth[length=10pt]}] (2, 2)-- (3, 2);
\draw [line width=1pt, -{Stealth[length=10pt]}] (3, 0)-- (3, 2);
\draw [line width=1pt, -{Stealth[length=10pt]}] (-1, 0)-- (-1, 2);
\draw [shift={(0, 1)}, line width=1pt, -{Stealth[length=10pt]}] plot[domain=2.356194490192345:3.9269908169872414, variable=\t]({1*1.4142135623730951*cos(\t r)+0*1.4142135623730951*sin(\t r)}, {0*1.4142135623730951*cos(\t r)+1*1.4142135623730951*sin(\t r)});
\draw [shift={(1, 1)}, line width=1pt, -{Stealth[length=10pt]}] plot[domain=0.7853981633974483:2.356194490192345, variable=\t]({1*1.4142135623730951*cos(\t r)+0*1.4142135623730951*sin(\t r)}, {0*1.4142135623730951*cos(\t r)+1*1.4142135623730951*sin(\t r)});
\draw [line width=1pt, -{Stealth[length=10pt]}] (2, 0)-- (3, 0);
\begin{scriptsize}
\draw [fill=black] (0, 0) circle (1.5pt);
\draw[color=black] (0, -0.2) node {$\ast$};
\draw [fill=black] (2, 0) circle (1.5pt);
\draw[color=black] (2, -0.2) node {$a(\ast)$};
\draw [fill=black] (2, 2) circle (1.5pt);
\draw[color=black] (2.25, 2.2) node {$ab(\ast)$};
\draw [fill=black] (0, 2) circle (1.5pt);
\draw[color=black] (-0.15, 2.2) node {$b(\ast)$};
\draw [fill=black] (-1, 0) circle (1.5pt);
\draw[color=black] (-1.2, -0.2) node {$ax(\ast)$};
\draw[color=black] (0.2, 1) node {$\ell_b$};
\draw[color=black] (1, 0.2) node {$\ell_a$};
\draw[color=black] (1.8, 1) node {$a \ell_b$};
\draw[color=black] (1, 1.8) node {$b \ell_a$};
\draw[color=black] (-0.5, 0.2) node {$\ell_x$};
\draw [fill=black] (-1, 2) circle (1.5pt);
\draw[color=black] (-1.2, 2.2) node {$xa(\ast)$};
\draw [fill=black] (3, 0) circle (1.5pt);
\draw[color=black] (3.2, -0.2) node {$ax(\ast)$};
\draw [fill=black] (3, 2) circle (1.5pt);
\draw[color=black] (3.2, 2.2) node {$bx(\ast)$};
\draw[color=black] (-0.5, 1.8) node {$b\ell_x$};
\draw[color=black] (2.5, 1.8) node {$ab\ell_x$};
\draw[color=black] (3.35, 1) node {$ax\ell_a$};
\draw[color=black] (-0.85, 1.45) node {$x \ell_a$};
\draw[color=black] (-1.75, 0.7) node {$xa\ell_a$};
\draw[color=black] (1, 2.55) node {$ab\ell_a$};
\draw[color=black] (2.5, 0.2) node {$a\tilde{\ell}$};
\draw[color=black] (1, 1) node {$\tilde{\varGamma}$};
\draw[color=black] (2.5, 1) node {$aH_3$};
\draw[color=black] (-0.5, 1) node {$H_3$};
\draw[color=black] (1, 2.2) node {$bH_4$};
\draw[color=black] (-1.22, 1) node {$xH_4$};
\end{scriptsize}
\end{tikzpicture}
    \caption{Gluing of $\tilde{\varGamma}$ and two copies of $H_3$ and $H_4$.}
    \label{fig:fig4}
\end{figure}
The projection $ \varGamma $ of $ \tilde{\varGamma}^e $ in $ B\Bgrouppos_n $ is a geometric cycle. Moreover, the domain of $ \tilde{\varGamma}^e $ can be deformed homeomorphically to a square such that the hypothesis of $ 1 $ is satisfied.
Its projection in $ B\Bgroup_n $ is homologous to $ \overline{\varGamma} $ mod $ 2 $ because the two copies of $ H_3 $ and $ H_4 $ cancel out.
\end{enumerate}
\end{proof}

\subsection{The cohomology of \texorpdfstring{$ \overline{\Fl}_3(\mathbb{R}) $}{text}} \label{coh3}
By the Serre spectral sequence, Theorem \ref{chso3} and Proposition \ref{cohomology Bpos}, the $E_2$ page of the Serre spectral sequence associated to the fibration 
\begin{equation} \label{ss3}
    SO(3) \rightarrow SO(3)_{h\Bgrouppos_3} \rightarrow B\Bgrouppos_3
\end{equation} is described in Figure~\ref{fig:spectral3}.
\newline
\begin{figure}
    \centering
    \begin{tikzpicture} 
 \matrix (m) [matrix of math nodes, 
 nodes in empty cells, nodes={minimum width=5ex, 
 minimum height=2ex, outer sep=-5pt}, 
 column sep=1ex, row sep=1ex]{
 &   &   &   &   &   &\\
 & 0 & 0 & 0 & 0 & \cdots \\
 & \mathbb{F}_2 & \mathbb{F}_2 & \mathbb{F}^2_2 & \mathbb{F}^3_2 & \cdots  \\
 & \mathbb{F}_2 &\mathbb{F}_2 & \mathbb{F}^2_2 & \mathbb{F}^3_2 & \cdots  \\
 & \mathbb{F}_2 &\mathbb{F}_2 & \mathbb{F}^2_2 & \mathbb{F}^3_2 & \cdots  \\
 & \mathbb{F}_2 &\mathbb{F}_2 & \mathbb{F}^2_2 & \mathbb{F}^3_2 & \cdots  \\
 \quad\strut & & & & & \strut \\};
 \draw[-stealth] (m-2-2.south east) -- (m-3-4.north west);
 \draw[-stealth](m-2-3.south east) -- (m-3-5.north west);
 \draw[-stealth] (m-3-2.south east) -- (m-4-4.north west);
 \draw[-stealth] (m-3-3.south east) -- (m-4-5.north west);
 \draw[-stealth] (m-4-2.south east) -- (m-5-4.north west);
 \draw[-stealth] (m-4-3.south east) -- (m-5-5.north west);
 \draw[-stealth] (m-5-2.south east) -- (m-6-4.north west);
 \draw[-stealth] (m-5-3.south east) -- (m-6-5.north west);
 \draw[thick] (m-1-1.east) -- (m-7-1.east) ;
\draw[thick] (m-7-1.north) -- (m-7-6.north) ;
\end{tikzpicture}
    \caption{$E_2$-page of the spectral sequence associated to \ref{ss3}.}
    \label{fig:spectral3}
\end{figure}
\begin{lemma} \label{d2bfl3}
$ d_2^{(0, 1)}(\beta_1) = \gamma_{1, 1}^2 \odot 1_1 + \alpha_2 \odot 1_1 $
\end{lemma}
\begin{proof}
First, $ H^2(B\Bgrouppos_3)= \mathbb{F}_2\{\alpha_2 \odot 1_1, \gamma_{1, 1}^2 \odot 1_1 \} $.
These two elements are restrictions of the Thom classes of the submanifolds $ \altmathcal{M}_1 $ and $ \altmathcal{M}_2 $ in $ B\Bgroup_3 $, respectively:
\begin{gather*}
 \altmathcal{M}_1 =\{[(x_1, x_2, x_3)] \in B(\Bgroup_3): \text{the 1\textsuperscript{st} coordinates of two of them are zero} \} \\
\altmathcal{M}_2 =\{[(x_1, x_2, x_3)] \in B(\Bgroup_3): \text{the 1\textsuperscript{st} and 2\textsuperscript{nd} coordinates of two of them are equal} \}
\end{gather*}
Second, we consider the immersion
\begin{gather*}
\overline{\varGamma}_1:[0, 1]^2 \to B(\Bgroup_3); \ \overline{\varGamma}_1(t_1, t_2)= (c_1(t_1), c_2(t_2), p_3), \\
\mathrm{where} \quad c_1(t)= (\cos(\pi t), \sin(\pi t), 0, 0, \cdots), \\
 c_2(t)= ( 2\cos(\pi t), 2 \sin(\pi t), 0, 0, \cdots), 
\end{gather*}
and $p_3$ is a fixed point far away from the origin.
This is depicted graphically below.
\begin{figure}[H]
    \centering
    \begin{tikzpicture}[line cap=round, line join=round, >=triangle 45, x=1cm, y=1cm]
\begin{axis}[
x=1cm, y=1cm, 
axis lines=middle, 
xmin=-3, 
xmax=5, 
ymin=-1, 
ymax=3, 
xtick=\empty, 
ytick=\empty, ]
\clip(-3, -1) rectangle (5, 3);
\draw [line width=1pt] plot[domain=0:3.141592653589793, variable=\t]({1*1*cos(\t r)+0*1*sin(\t r)}, {0*1*cos(\t r)+1*1*sin(\t r)});
\draw [line width=1pt] plot[domain=0:3.141592653589793, variable=\t]({1*2*cos(\t r)+0*2*sin(\t r)}, {0*2*cos(\t r)+1*2*sin(\t r)});
\node at (0.4, -0.2) {\tiny $(0, 0)$};
\node at (0.35, 0.75) {\footnotesize $c_1$};
\node at (0.95, 1.45) {\footnotesize $c_2$};
\node at (3.95, -0.3) {\footnotesize $p_3$};
\begin{scriptsize}
\draw [fill=black] (4, 0) circle (1.5pt);
\end{scriptsize}
\end{axis}
\end{tikzpicture}
    \caption{Graphical depiction of the immersion $\varGamma_1$.}
    \label{fig:fig6}
\end{figure}

$\partial \overline{\varGamma}_1=0$ and the hypotheses of Lemma \ref{geometric machinery d2}.3 are satisfied with $ x = (1, 2) \in \mathcal{S}_3 \subseteq \Bgroup_3 $. The chain $ \varGamma_1 $ in $ B\Bgrouppos_3 $ provided by the lemma represents a homology class $ C_1 \in H_2(B\Bgrouppos_3) $ that transgresses to $ [SO(2)] $ by Lemma \ref{geometric machinery d2}.2. By \cite[Theorem 6.8]{JM01}, we deduce that $ d^2_{(2, 0)}(C_1) = [SO(2)] $.

We also define
\begin{align*}
    \varGamma_2 \colon S^2_+ &\to B(\Bgrouppos_3),  \\
(x_1, x_2, x_3) &\mapsto [((x_1, x_2, x_3, 0, \dots), (3+x_1, x_2, x_3, 0, \dots), (3-x_1, -x_2, -x_3, \dots))], 
\end{align*}
where $ S^2_+ $ is the half-sphere $ \{(x_1, x_2, x_3) \in \mathbb{R}^3: x_1^2 + x_2^2+x_3^2 = 1, x_3 \geq 0 \} $. See Figure~\ref{fig:fig7} for a graphical representation of $ \varGamma_2 $.
\begin{figure}
    \centering
    \begin{tikzpicture}[line cap=round, line join=round, x=1cm, y=1cm]
\begin{axis}[
x=1cm, y=1cm, 
axis lines=middle, 
xmin=-2, 
xmax=5, 
ymin=-2, 
ymax=2, 
xtick=\empty, 
ytick=\empty, ]
\draw [line width=1pt,->] plot[domain=-1.5:1.5, variable=\t]({\t}, {\t});
\draw [line width=1pt,dashed] plot[domain=0:3.141592653589793, variable=\t]({1*1*cos(\t r)+0*1*sin(\t r)}, {0*1*cos(\t r)+1*1*sin(\t r)});
\draw [line width=1pt,dashed] plot[domain=0:6.283185307, variable=\t]({1*1*cos(\t r)+0*0.3*sin(\t r)}, {0*1*cos(\t r)+1*0.3*sin(\t r)});
\draw [line width=1pt,dashed] plot[domain=0:6.283185307, variable=\t]({3+1*1*cos(\t r)+0*1*sin(\t r)}, {0*1*cos(\t r)+1*1*sin(\t r)});
\draw [line width=1pt,dashed] plot[domain=0:6.283185307, variable=\t]({3+1*1*cos(\t r)+0*0.3*sin(\t r)}, {0*1*cos(\t r)+1*0.3*sin(\t r)});

\node at (0.4, -0.2) {\tiny $(0, 0, 0)$};
\begin{scriptsize}
\draw [fill=black] (3/5,4/5) circle (2pt);
\draw [fill=black] (3+3/5,4/5) circle (2pt);
\draw [fill=black] (3-3/5,-4/5) circle (2pt);
\end{scriptsize}
\end{axis}
\end{tikzpicture}
    \caption{Graphical representation of $ \varGamma_2 $.}
    \label{fig:fig7}
\end{figure}
It is a mod $ 2 $ cycle and its homology class $ C_2 = [\varGamma_2] $ is such that $ d_2^{(2, 0)}(C_2) = [SO(2)] $ by Lemma \ref{geometric machinery d2}.2 and \cite[Theorem 6.8]{JM01}.

Third, we determine the pairing between $ C_1 $, $ C_2 $, and our cohomology classes.
Since $ \rho_3 $ is surjective in degree $ 2 $, by the naturality of the homology-cohomology pairing, we can perform calculations in $ B\Bgroup_3 $ instead of $ B\Bgrouppos_3 $. Here the representing submanifolds $ \altmathcal{M}_1 $ and $ \altmathcal{M}_2 $ are transverse to $ \overline{\varGamma}_1 $ and $ \overline{\varGamma}_2 $, hence we can compute the pairing by counting the number of intersections.
Explicitly
\[
\langle \alpha_2 \odot 1_1, C_1 \rangle_{\Bgrouppos_3} = \langle \alpha_2 \odot 1_1), [\overline{\varGamma}_1] \rangle_{\Bgrouppos_3} = \langle [\altmathcal{M}_1], [\overline{\varGamma}_1] \rangle_{\Bgroup_3} = \# \overline{\varGamma}_1^{-1}(\altmathcal{M}_1) = 1
\]
and similarly $ \langle \gamma_{1, 1}^2 \odot 1_1, C_1 \rangle = 0 $, $ \langle \alpha_2 \odot 1_1, C_2 \rangle = 0 $ and $ \langle \gamma_{1, 1}^2 \odot 1_1, C_2 \rangle = 1 $.
Since $ d_2^{(0, 1)} $ is the linear dual of $ d^2_{(2, 0)} $ and $ \langle \beta_1, [SO(2)] \rangle_{SO(n)} = 1 $, we deduce the lemma from the calculations of $ d_2(C_1) $ and $ d_2(C_2) $.
\end{proof}

Multiplicatively, $ E_2 $ is generated by the bottom row, on which $ d_2 $ is zero by degree reasons, and $ \beta_1 $. Therefore, $ d_2 $ is fully determined by Lemma \ref{d2bfl3}.
The $E_3$ page is given as follows
\begin{figure}[H]
    \centering
    \begin{tikzpicture}
 \matrix (m) [matrix of math nodes, 
 nodes in empty cells, nodes={minimum width=5ex, 
 minimum height=2ex, outer sep=-5pt}, 
 column sep=1ex, row sep=1ex]{
 & & & & & & \\
 & 0 & 0 & 0 & 0 & \cdots \\
 &\mathbb{F}_2 & \mathbb{F}_2 & \mathbb{F}^2_2 & \mathbb{F}^2_2 & \cdots \\
 &0 & 0 & 0 & 0 & \cdots \\
 &\mathbb{F}_2 & \mathbb{F}_2 & \mathbb{F}^2_2 & \mathbb{F}^2_2 & \cdots \\
 \quad\strut & & & & & \strut \\};
 \draw[-stealth] (m-2-2.south east) -- (m-4-5.north west);
 \draw[-stealth] (m-3-2.south east) -- (m-5-5.north west);
 \draw[-stealth] (m-3-3.south east) -- (m-5-6.north west);
 \draw[thick] (m-1-1.east) -- (m-6-1.east) ;
\draw[thick] (m-6-1.north) -- (m-6-6.north) ;
\end{tikzpicture}
    \caption{$E_3$-page of the spectral sequence associated to \ref{ss3}.}
    \label{fig:spectral3p3}
\end{figure}

\begin{lemma} \label{d3bfl3}
$ d_3^{(0, 2)}(\beta_1^2) = \alpha_3 $.
\end{lemma}
\begin{proof}
As $ \beta_1 $ is $ 1 $-dimensional, $ \beta_1^2 = \Sq^1(\beta_1) $. Hence, by \cite[Corollary 6.9]{JM01}, 
\[
d_3^{(0, 2)}(\beta_1^2) = \Sq^1(d_2^{(0, 1)}(\beta_1) =\Sq^1(\alpha_2 \odot 1_1 + \gamma_{1, 1} \odot \alpha_1).
\]
Since the restriction map $ \rho_3 \colon H^*(B\Bgroup_3) \to H^*(B\Bgrouppos_3) $ preserve the Steenrod algebra action, $ \Sq^1(\alpha_2 \odot 1_1 + \gamma_{1, 1} \odot \alpha_1) $ can be computed in $ H^*(B\Bgroup_3) $ instead of $ B(\Bgrouppos_3) $. The Steenrod squares are computed in \cite[Theorem 8.2]{LG21}. This yields
\begin{gather*}
\Sq^1(\gamma_{1, 1} \odot \alpha_1 + \alpha_2 \odot 1_1) = \gamma_{1, 1} \alpha_2 \odot 1_1 + \alpha_1^2 \odot \alpha_1 \odot 1_1 + \gamma_{1, 1}^2 \odot \alpha_1 + \gamma_{1, 1} \odot \alpha_1^2 \\
= (\gamma_{1, 1} \odot 1_1 + \alpha_1 \odot 1_2) \cdot (\alpha_2 \odot 1_1 + \gamma_{1, 1} \odot \alpha_1) + \alpha_3.
\end{gather*}
Noting that $ \gamma_{1, 1} \odot 1_1 + \alpha_1 \odot 1_2 $ is zero in the cohomology of $ B\Bgrouppos_3 $ completes the proof.
\end{proof}

By Lemma \ref{d3bfl3}, $ d_3 $ on row $ 2 $ is the multiplication by $ \alpha_3 $, which is injective. Hence, the fourth page of the spectral sequence consists only of the bottom row, equal to $ H^*(B\Bgrouppos_3)/(\gamma_{1, 1} \odot \alpha_1 + \alpha_2 \odot 1_1, \alpha_3) $. $ \overline{\Fl}_3(\mathbb{R}) $ being $ 3 $-dimensional, this quotient is zero in degree greater than $3$. For degree at most $ 3 $, the quotient is explicitly computed from Theorem \ref{cohob}. In conclusion, the above arguments prove the following.
\begin{theorem} \label{main:n=3}
There is a ring isomorphism
\[
H^*(\overline{\Fl}_3(\mathbb{R})) \cong \frac{H^*(B\Bgrouppos_3)}{(\gamma_{1, 1} \odot 1_1 + \alpha_1 \odot 1_2, \gamma_{1, 1} \odot \alpha_1 + \alpha_2 \odot 1_1, \alpha_3)} \cong \frac{\mathbb{F}_2[x]}{(x^4)}, 
\]
where $ x $ is the pullback of $ \gamma_{1, 1} \in H^*(B\Bgroup_3) $ in the cohomology of $ \overline{\Fl}_3(\mathbb{R}) $.
\end{theorem}

\begin{corollary} \label{Poincare3}
The Poincar\'e series of $ H^*(\overline{\Fl}_3(\mathbb{R})) $ is $ \Pi_{\overline{\Fl}_3(\mathbb{R})}(t) = t^3 + t^2 + t + 1 $.
\end{corollary}

\begin{remark} It is worth noticing that although the cohomology ring of $ \overline{\Fl}_3(\mathbb{R}) $ and the $\mathbb{R}P^3$ are isomorphic, these spaces are not homeomorphic as they have non-isomorphic first homotopy groups. 
\end{remark}

\subsection{The cohomology of \texorpdfstring{$\overline{\Fl}_4(\mathbb{R})$}{text}} \label{coh4}

The second page of the Serre spectral sequence of the fibration 
\begin{equation} \label{ss4}
    SO(4) \hookrightarrow SO(4)_{h\Bgrouppos_4} \twoheadrightarrow B\Bgrouppos_4
\end{equation} is shown below.
\begin{figure}[H]
    \centering
    \begin{tikzpicture}
 \matrix (m) [matrix of math nodes, 
 nodes in empty cells, nodes={minimum width=5ex, 
 minimum height=2ex, outer sep=-5pt}, 
 column sep=1ex, row sep=1ex]{ \quad\strut
 & & & & & & & \\
 & 0 & 0 & 0 & 0 & \cdots & \\
 & \mathbb{F}_2 & \mathbb{F}_2 & \mathbb{F}^3_2 & \mathbb{F}^6_2 &\cdots & \\
 & \mathbb{F}_2 & \mathbb{F}_2 & \mathbb{F}^3_2 & \mathbb{F}^6_2 &\cdots & \\
 & \mathbb{F}_2 &\mathbb{F}_2 & \mathbb{F}^3_2 & \mathbb{F}^6_2 &\cdots & \\
 & \mathbb{F}_2 \oplus \mathbb{F}_2 &\mathbb{F}_2 \oplus \mathbb{F}_2 & \mathbb{F}^3_2 \oplus \mathbb{F}^3_2 & \mathbb{F}^6_2 \oplus \mathbb{F}^4_2 & \cdots & \\
 & \mathbb{F}_2 &\mathbb{F}_2 & \mathbb{F}^3_2 & \mathbb{F}^6_2 & \cdots & \\
 & \mathbb{F}_2 &\mathbb{F}_2 & \mathbb{F}^3_2 & \mathbb{F}^6_2 &\cdots & \\
 & \mathbb{F}_2 &\mathbb{F}_2 & \mathbb{F}^3_2 & \mathbb{F}^6_2 &\cdots & \\
 \quad \strut & & & & & & & \strut \\};
 \draw[-stealth] (m-2-2.south east) -- (m-3-4.north west);
 \draw[-stealth](m-2-3.south east) -- (m-3-5.north west);
 \draw[-stealth] (m-3-2.south east) -- (m-4-4.north west);
 \draw[-stealth] (m-3-3.south east) -- (m-4-5.north west);
 \draw[-stealth] (m-4-2.south east) -- (m-5-4.north west);
 \draw[-stealth] (m-4-3.south east) -- (m-5-5.north west);
 \draw[-stealth] (m-5-2.south east) -- (m-6-4.north west);
 \draw[-stealth] (m-5-3.south east) -- (m-6-5.north west);
 \draw[-stealth] (m-6-2.south east) -- (m-7-4.north west);
 \draw[-stealth] (m-6-3.south east) -- (m-7-5.north west);
 \draw[-stealth] (m-7-2.south east) -- (m-8-4.north west);
 \draw[-stealth] (m-7-3.south east) -- (m-8-5.north west);
 \draw[thick] (m-1-1.east) -- (m-10-1.east) ;
 \draw[-stealth] (m-8-2.south east) -- (m-9-4.north west);
 \draw[-stealth] (m-8-3.south east) -- (m-9-5.north west);
 \draw[thick] (m-1-1.east) -- (m-10-1.east) ;
\draw[thick] (m-10-1.north) -- (m-10-8.north) ;
\end{tikzpicture}
    \caption{$E_2$-page of the spectral sequence associated to \ref{ss4}.}
    \label{fig:spectral42}
\end{figure}

\begin{lemma}\label{d2bfl4}
$ d_2^{(0, 1)}(\beta_1) = \gamma_{1, 1} \odot \alpha_1 \odot 1_1 + \alpha_2 \odot 1_2 + \gamma_{1, 2} $.
\end{lemma}
\begin{proof}
The fibration \ref{ss3} is naturally a sub-fibration of \ref{ss4}. By comparison of the associated Serre spectral sequences, in the latter $ d_2^{(0, 1)}(\beta_1) $ is determined by Lemma \ref{d2bfl3} modulo the kernel $ K $ of the restriction $ H^2(B\Bgrouppos_4) \to H^2(B\Bgrouppos_3) $.
Since $ \rho_3 $ and $ \rho_4 $ are surjective in degree $ 2 $, this kernel can be determined via the restriction map $ H^*(B\Bgroup_4) \to H^*(B\Bgroup_3) $, which is the composition of the component of the coproduct $ \Delta_{(3, 1)} \colon H^*(B\Bgroup_4) \to H^*(B\Bgroup_3) \otimes H^*(B\Bgroup_1) $ an the projection on $ H^*(B\Bgroup_3) $.
Computing it on the classes listed in Proposition \ref{cohomology Bpos} shows that $ K $ is generated by $ \gamma_{1, 2} $.
We deduce that
\[
d_2^{(0, 1)}(\beta_1) = \gamma_{1, 1} \odot \alpha_1 \odot 1_1 + \alpha_2 \odot 1_2 + \lambda\gamma_{1, 2}
\]
for some $ \lambda \in \mathbb{F}_2 $.

To determine $ \lambda $, we pair $ \beta_1 $ with the dual homological differential $ d^2_{(2, 0)} $ of a geometric cycle that pairs non-trivially with $ \gamma_{1, 2} $.
We consider the immersion
\begin{gather*}
\overline{\varGamma}_3 \colon [0, 1]^2 \to B\Bgroup_4, \ \overline{\varGamma}_3(t_1, t_2) = [(2+\cos(\pi t_1), \sin(\pi t_1), 0, \dots), (2 - \cos(\pi t_1), -\sin(\pi t_1), 0, \dots), \\
(5 + \cos(\pi t_2), \sin(\pi t_2), 0, \dots), (5 - \cos(\pi t_2), -\sin(\pi t_2), 0, \dots)], 
\end{gather*}
which is a geometric cycle in $ B\Bgroup_4 $. It is represented below.
\begin{figure}[H]
    \centering
    \begin{tikzpicture}[line cap=round, line join=round, x=1cm, y=1cm]
\begin{axis}[
x=1cm, y=1cm, 
axis lines=middle, 
xmin=-0.5, 
xmax=7, 
ymin=-2, 
ymax=2, 
xtick=\empty, 
ytick=\empty, ]
\draw [line width=1pt] plot[domain=0:6.283185307, variable=\t]({2+1*1*cos(\t r)+0*1*sin(\t r)}, {0*1*cos(\t r)+1*1*sin(\t r)});
\draw [line width=1pt] plot[domain=0:6.283185307, variable=\t]({5+1*1*cos(\t r)+0*1*sin(\t r)}, {0*1*cos(\t r)+1*1*sin(\t r)});

\node at (0.4, -0.2) {\tiny $(0, 0)$};
\begin{scriptsize}
\draw [fill=black] (2+3/5,4/5) circle (2pt);
\draw [fill=black] (2-3/5,-4/5) circle (2pt);
\draw [fill=black] (5+3/5,4/5) circle (2pt);
\draw [fill=black] (5-3/5,-4/5) circle (2pt);
\end{scriptsize}
\end{axis}
\end{tikzpicture}
    \caption{Graphical representation of $\overline{\varGamma}_3$.}
    \label{fig:fig10}
\end{figure}

By Lemma \ref{geometric machinery d2}.3, we obtain a geometric cycle $ \varGamma_3 $ in $ B\Bgrouppos_4 $ satisfying the hypotheses of Lemma \ref{geometric machinery d2}.2, with $ y = (1, 3)(2, 4) \in \mathcal{S}_4 \subseteq \Bgroup_4 $. Therefore $ d^2_{(2, 0)}[\varGamma_3] = [SO(2)] $. It pairs non-trivially with $ \beta_1 $, thus $ \langle [\varGamma_3], d_2^{(0, 1)}(\beta_1) \rangle = 1 $.

Moreover, $\gamma_{1, 2}$ can be represented as the Thom class associated to the submanifold of $ B\Bgroup_4 $
\[ \altmathcal{M}_3 = \{[(p_1, p_2, p_3, p_4)] : p_{1, 1}= p_{2, 1}, p_{3, 1}= p_{4, 1} \}. \]
Since $ \overline{\varGamma_3} $ intersects it transversally at a single point, $ \langle [\varGamma_3], \gamma_{1, 2} \rangle = 1 $ by the naturality of the pairing.
This determines $ \lambda $, which is equal to $ 1 $.
\end{proof}

We can deduce the following with the same argument used to prove Lemma \ref{d3bfl3}.
\begin{lemma} \label{d3bfl4}
$ \beta_1^2 $ is transgressive and $ d_3^{(0, 2)}(\beta_1^2) = \gamma_{1, 1} \odot \alpha_2 + \alpha_3 \odot 1_1 + \gamma_{2, 1} $.
\end{lemma}

We exploit Lemma \ref{d3bfl4} to determine $ d_2^{(0, 3)}(\beta_3) $.
\begin{lemma} \label{b3 transgressive}
$ d_2^{(0, 3)}(\beta_3) = 0 $.
\end{lemma}
\begin{proof}
By comparison with the Serre spectral sequence of the fibration \ref{ss3}, the same argument used in the first part of the proof of Lemma \ref{d2bfl4} shows that $ d_2^{(0, 3)}(\beta_3) = \lambda' \gamma_{1, 2} \otimes \beta_1^2 $ for some $ \lambda' \in \mathbb{F}_2 $.

By the multiplicativity of the spectral sequence, $ d_2^{(2, 2)}(\gamma_{1, 2} \otimes \beta_1^2) = 0 $. Hence, $ d_3^{(2, 0)}(\gamma_{1, 2} \otimes \beta_1^2) $ is defined. Again by the multiplicativity and Lemma \ref{d3bfl4}, this is equal to \[ d_3^{(0, 2)}(\beta_1^2) \cdot \gamma_{1, 2} = \gamma_{2, 1}\gamma_{1, 2} + \gamma_{1, 1}^2 \odot \gamma_{1, 1}\alpha_2 ,\] which is different from $ 0 $ in $ H^*(B\Bgroup_4)/(\gamma_{1, 1} \odot 1_2 + \alpha_1 \odot 1_3, \gamma_{1, 1} \odot \alpha_1 \odot 1_1 + \alpha_2 \odot 1_2 + \gamma_{1, 2}) $, considered as a subspace of the bottom row of the $ E_3 $-page. Thus, $ \gamma_{1, 2} \otimes \beta_1^2 $ must be non-zero in $ E_3^{(2, 2)} $, and this forces $ \lambda' = 0 $.
\end{proof}

We can now describe the spectral sequence entirely. The $ E_2 $-page is multiplicatively generated by the bottom row, on which $ d_2 $ is zero, and the classes $ \beta_1 $ and $ \beta_3 $. Therefore, Lemmas \ref{d2bfl4} and \ref{b3 transgressive} completely determine $d_2$. The first four columns of the third page look like the following.
\begin{figure}[H]
\centering
\begin{tikzpicture}
 \matrix (m) [matrix of math nodes, 
 nodes in empty cells, nodes={minimum width=5ex, 
 minimum height=2ex, outer sep=-5pt}, 
 column sep=1ex, row sep=1ex]{ \quad\strut
 & & & & & & & \\
 & 0 & 0 & 0 & 0 & \dots & \\
 & \mathbb{F}_2 & \mathbb{F}_2 & \mathbb{F}^2_2 & \mathbb{F}^5_2 &\cdots & \\
 & 0 & 0 & 0 & 0 &\cdots & \\
 & \mathbb{F}_2 & \mathbb{F}_2 & \mathbb{F}^3_2 & \mathbb{F}^6_2 & \cdots & \\
 & \mathbb{F}_2 &\mathbb{F}_2 & \mathbb{F}^2_2 & \mathbb{F}^5_2 & \cdots & \\
 & 0 & 0 & 0 & 0 &\cdots & \\
 & \mathbb{F}_2 &\mathbb{F}_2 & \mathbb{F}^2_2 & \mathbb{F}^5_2 &\cdots & \\
 \quad \strut & & & & & & & \strut \\};
 \draw[-stealth] (m-2-2.south east) -- (m-4-5.north west);
 \draw[-stealth] (m-3-2.south east) -- (m-5-5.north west);
 \draw[-stealth] (m-4-2.south east) -- (m-6-5.north west);
\draw[-stealth] (m-5-2.south east) -- (m-7-5.north west) ;
\draw[-stealth] (m-6-2.south east) -- (m-8-5.north west) ;
 \draw[thick] (m-1-1.east) -- (m-9-1.east) ;
\draw[thick] (m-9-1.north) -- (m-9-8.north) ;
6-6.north west);
\end{tikzpicture}
\caption{$E_3$-page of the spectral sequence associated to \ref{ss4}.}
\label{spectral43}
\end{figure}
The entries in the rows $ 1 $, $ 4 $, and $ 6 $ are zero in the first four columns by direct computations using Proposition \ref{cohomology Bpos}.

$ d_3 $ is determined on all $ E_3 $ multiplicatively from $ d_3^{(0, 2)}(\beta_1^2) $, that we have computed above, and $ d_3^{(0, 3)}(\beta_3) = 0 $ by degree reasons.
Consequently, on the fourth page of the spectral sequence we have 
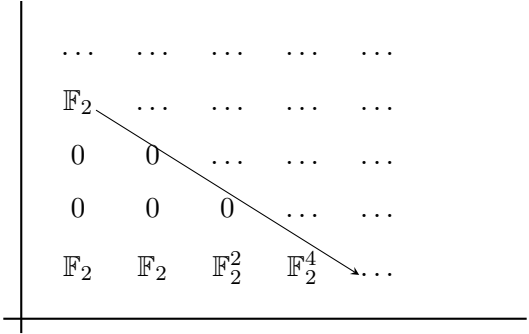
\begin{figure}[H]
    \centering
    \begin{tikzpicture} \label{spectral44}
 \matrix (m) [matrix of math nodes, 
 nodes in empty cells, nodes={minimum width=5ex, 
 minimum height=2ex, outer sep=-5pt}, 
 column sep=1ex, row sep=1ex]{ \quad\strut
 & & & & & & & \\
 & \dots & \dots & \dots & \dots & \dots & \\
 & \mathbb{F}_2 & \dots & \dots & \dots & \dots & \\
 & 0 & 0 & \dots & \dots & \dots & \\
 & 0 & 0 & 0 & \dots &\dots & \\
 & \mathbb{F}_2 &\mathbb{F}_2 & \mathbb{F}^2_2 & \mathbb{F}^4_2 &\dots & \\
 \quad \strut & & & & & & & \strut \\};
 \draw[-stealth] (m-3-2.south east) -- (m-6-6.north west);
 \draw[thick] (m-1-1.east) -- (m-7-1.east) ;
\draw[thick] (m-7-1.north) -- (m-7-8.north) ;
\end{tikzpicture}
    \caption{$E_4$-page of the spectral sequence associated to \ref{ss4}.}
    \label{fig:spectral44}
\end{figure}

\begin{lemma} \label{d4bfl4}
$ d_4^{(0, 3)}(\beta_3) \not= 0 $.
\end{lemma}
\begin{proof}
We fix a unit vector $ v \in S^3 \subseteq \mathbb{R}^4 $ and we consider the map $ ev \colon SO(4) \to S^3 $ that evaluates a transformation in $ SO(4) $ at $ v $. The fiber of this map is the subgroup $ \Stab_{SO(4)}(v) \cong SO(3) $. After taking the homotopy quotient, $ ev $ induces a morphism $ (E\Bgroup_4 \times SO(4))/\Bgrouppos_4 \to S^3/\Bgrouppos_4 $. If we choose $ v = (1, 0, 0, 0) $, then $ \Stab_{\Bgrouppos_4}(v) = \Bgrouppos_4 \cap \Stab_{SO(4)}(v) \cong \Bgrouppos_3 $. Consequently, the following diagram commutes and the top and bottom sequences are fibrations:
\begin{center}
\begin{tikzcd}
SO(3) \arrow[hook]{r} \arrow{d} & SO(4) \arrow[two heads]{r} \arrow{d} & S^3 \arrow{d} \\
\frac{E\Bgroup_4 \times SO(3)}{\Bgrouppos_3} \arrow[hook]{r} & \frac{E\Bgroup_4 \times SO(4)}{\Bgrouppos_4} \arrow[two heads]{r} & \frac{S^3}{\Bgrouppos_4}
\end{tikzcd}
\end{center}

The map in cohomology $ H^*(S^3/\Bgrouppos_4) \to H^*(S^3) $ is zero in positive degrees. Moreover, the restriction of every cohomology class of $ SO(3)_{h\Bgrouppos_3} $ to $ H^*(SO(3)) $ is zero in positive degrees because in the spectral sequence analyzed in \S \ref{coh3} nothing survives in the leftmost column of the limit page.
Therefore, by comparison of the Serre spectral sequence of the top and bottom fibrations, we deduce that the restriction map $ H^3(SO(4)_{h\Bgrouppos_4}) \to H^3(SO(4)) $ is zero. 
In particular, in bidegree $ (0, 3) $ of the mod $ 2 $ cohomological Serre spectral sequence of the fibration \ref{ss4}, the $ E_\infty $ page must be zero. This forces $ d_4^{(0, 3)}(\beta_3) \not= 0 $.
\end{proof}

This determines the page $ E_5 $. Its first four columns are concentrated in the bottom row. Since no further differentials meet these columns, they survive to the limit page and yield $ H^*(\overline{\Fl}_4(\mathbb{R})) $ up to degree $ 3 $.
Since $ \overline{\Fl}_4(\mathbb{R}) $ is $ 6 $-dimensional, Poincar\'e duality determines all the remaining cohomology groups.
In conclusion, we have just proved the following result.
\begin{theorem} \label{poincare4}
The Poincaré series of $ H^*(\overline{\Fl}_4(\mathbb{R}); \mathbb{F}_2) $ is 
\[
\Pi_{\overline{\Fl}_4(\mathbb{R})}(t) = t^6 + t^5 + 2t^4 + 4t^3 + 2t^2 + t + 1.
\]
\end{theorem}

\subsection{The cohomology of \texorpdfstring{$\overline{\Fl}_5(\mathbb{R})$} {text}} \label{coh5}


\begin{lemma} \label{d2bfl5}
In the Serre spectral sequence associated to \begin{equation} \label{ss5}
    SO(5) \hookrightarrow \overline{\Fl}_5(\mathbb{R}) \twoheadrightarrow B\Bgrouppos_5
\end{equation} the differential $d_2$ has the following form:
\[
d_2^{(0, 1)}(\beta_1) = \gamma_{1, 2} \odot 1_1 + \gamma_{1, 1} \odot \alpha_1 \odot 1_2 + \alpha_2 \odot 1_3 \quad \text{and} \quad d_2^{(0, 3)}(\beta_3) = 0.
\]
\end{lemma}
\begin{proof}
The restriction map $ H^d (B\Bgrouppos_5) \to H^d (B\Bgrouppos_4) $ is an isomorphism for $ d = 2 $. Therefore the differential of the leftmost column of the second page is entirely determined by comparison with the spectral sequence for $ SO(4)_{h\Bgrouppos_4} $ analyzed in \S \ref{coh4} and the result follows from Lemmas \ref{d2bfl4} and \ref{b3 transgressive}.
\end{proof}

$ d_3^{(0, 3)}(\beta_3) = 0 $ because $ E_3^{(3, 1)} = 0 $.
Since $ \beta_1^2 = \Sq^1(\beta_1) $ and $ \beta_1^4 = \Sq^2(\beta_1^2) $, the following result is proved similarly as Lemma \ref{d3bfl3}.
\begin{lemma} \label{d35bfl5}
\[
d_3^{(0, 2)}(\beta_1^2) = \gamma_{1, 2} \odot \alpha_1 + \gamma_{1, 1} \odot \alpha_2 + \gamma_{2, 1} \odot 1_1 + \alpha_3 \quad \text{and} \quad  d_5^{(0, 4)}(\beta_1^4) = \alpha_5.
\]
\end{lemma}
Finally, since the restriction of $ d_4^{(0, 4)}(\beta_3)$ to $ B\Bgrouppos_4 $ is nonzero by Lemma \ref{d4bfl4}, so is $ d_4^{(0, 4)}(\beta_3)$. Moreover, it follows from a straightforward calculation that the multiplication by $ \gamma_{1, 1} \odot 1_3 $ induces an injective map $ E_4^{(4, 0)} \to E_4^{(5, 0)} $. Thus, $ d_4^{(1, 4)}(\beta_3 \otimes (\gamma_{1, 1} \odot 1_3)) \not= 0 $.

These calculations for the first few pages of the spectral sequence are summarized in the following spectral sequence diagrams. 

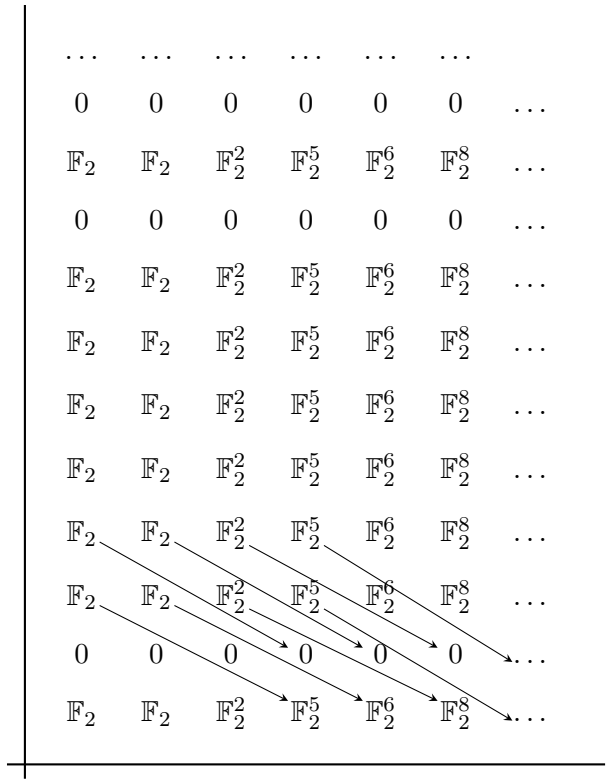
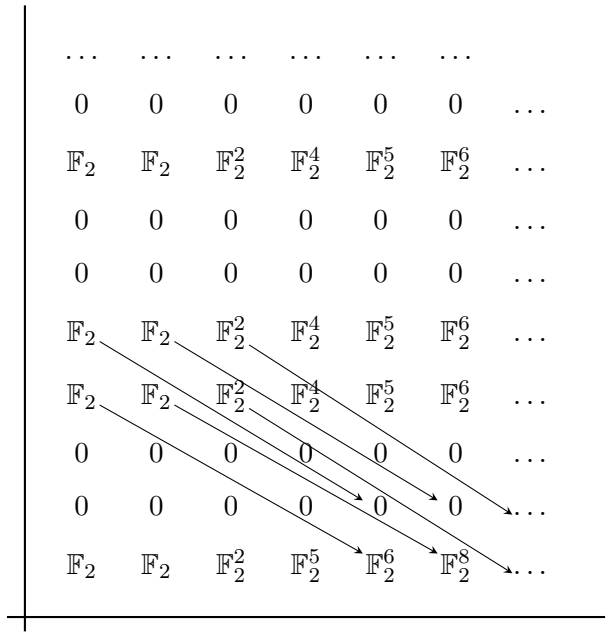
\begin{figure}[H]
    \centering
    \begin{tikzpicture} 
\matrix (m) [matrix of math nodes, 
 nodes in empty cells, nodes={minimum width=5ex, 
 minimum height=2ex, outer sep=-5pt}, 
 column sep=1ex, row sep=1ex]{ \quad\strut
 & & & & & & & & \\
 & \cdots & \cdots & \cdots & \cdots & \cdots & \cdots & \cdots & \\
 & \mathbb{F}_2 \oplus \mathbb{F}_2 & \mathbb{F}_2 \oplus \mathbb{F}_2 & \mathbb{F}_2^3 \oplus \mathbb{F}_2^3 & \mathbb{F}_2^6 \oplus \mathbb{F}_2^6 & \mathbb{F}_2^9 \oplus \mathbb{F}_2^9 & \mathbb{F}_2^{14} \oplus \mathbb{F}_2^{14} & \cdots & \\
 & \mathbb{F}_2 \oplus \mathbb{F}_2 & \mathbb{F}_2 \oplus \mathbb{F}_2 & \mathbb{F}_2^3 \oplus \mathbb{F}_2^3 & \mathbb{F}_2^6 \oplus \mathbb{F}_2^6 & \mathbb{F}_2^9 \oplus \mathbb{F}_2^9 & \mathbb{F}_2^{14} \oplus \mathbb{F}_2^{14} & \cdots & \\
 & \mathbb{F}_2 \oplus \mathbb{F}_2 & \mathbb{F}_2 \oplus \mathbb{F}_2 & \mathbb{F}_2^3 \oplus \mathbb{F}_2^3 & \mathbb{F}_2^6 \oplus \mathbb{F}_2^6 & \mathbb{F}_2^9 \oplus \mathbb{F}_2^9 & \mathbb{F}_2^{14} \oplus \mathbb{F}_2^{14} & \cdots & \\
 & \mathbb{F}_2 & \mathbb{F}_2 & \mathbb{F}_2^3 & \mathbb{F}_2^6 & \mathbb{F}_2^9 & \mathbb{F}_2^{14} & \cdots & \\
 & \mathbb{F}_2 & \mathbb{F}_2 & \mathbb{F}_2^3 & \mathbb{F}_2^6 & \mathbb{F}_2^9 & \mathbb{F}_2^{14} & \cdots & \\
 & \mathbb{F}_2 & \mathbb{F}_2 & \mathbb{F}_2^3 & \mathbb{F}_2^6 & \mathbb{F}_2^9 & \mathbb{F}_2^{14} & \cdots & \\
 \quad \strut & & & & & & & & \strut \\};
 
 \draw[-stealth] (m-7-2.south east) -- (m-8-4.north west);
 \draw[-stealth] (m-7-3.south east) -- (m-8-5.north west);
 \draw[-stealth] (m-7-4.south east) -- (m-8-6.north west);
 \draw[-stealth] (m-7-5.south east) -- (m-8-7.north west);
 \draw[-stealth] (m-6-2.south east) -- (m-7-4.north west);
 \draw[-stealth] (m-6-3.south east) -- (m-7-5.north west);
 \draw[-stealth] (m-6-4.south east) -- (m-7-6.north west);
 \draw[-stealth] (m-6-5.south east) -- (m-7-7.north west);
 \draw[-stealth] (m-5-2.south east) -- (m-6-4.north west);
 \draw[-stealth] (m-5-3.south east) -- (m-6-5.north west);
 \draw[-stealth] (m-5-4.south east) -- (m-6-6.north west);
 \draw[-stealth] (m-5-5.south east) -- (m-6-7.north west);

 \draw[-stealth] (m-4-2.south east) -- (m-5-4.north west);
 \draw[-stealth] (m-4-3.south east) -- (m-5-5.north west);
 \draw[-stealth] (m-4-4.south east) -- (m-5-6.north west);
 \draw[-stealth] (m-4-5.south east) -- (m-5-7.north west);

 \draw[-stealth] (m-3-2.south east) -- (m-4-4.north west);
 \draw[-stealth] (m-3-3.south east) -- (m-4-5.north west);
 \draw[-stealth] (m-3-4.south east) -- (m-4-6.north west);
 \draw[-stealth] (m-3-5.south east) -- (m-4-7.north west);
 \draw[thick] (m-1-1.east) -- (m-9-1.east) ;
 \draw[thick] (m-9-1.north) -- (m-9-9.north) ;
\end{tikzpicture}
    \caption{$E_2$-page of the spectral sequence associated to \ref{ss5}.}
    \label{fig:spectral52}
\end{figure}

\begin{figure}[H]
    \centering
    \begin{tikzpicture}
\matrix (m) [matrix of math nodes, 
 nodes in empty cells, nodes={minimum width=5ex, 
 minimum height=2ex, outer sep=-5pt}, 
 column sep=1ex, row sep=1ex]{ \quad\strut
 & & & & & & & & \\
 & \cdots & \cdots & \cdots & \cdots & \cdots & \cdots & \cdots & \\
 & \mathbb{F}_2 & \mathbb{F}_2 & \mathbb{F}_2^2 & \mathbb{F}_2^5 & \mathbb{F}_2^6 & \mathbb{F}_2^8 & \dots & \\
 & \mathbb{F}_2 & \mathbb{F}_2 & \mathbb{F}_2^2 & \mathbb{F}_2^5 & \mathbb{F}_2^6 & \mathbb{F}_2^8 & \dots & \\
 & \mathbb{F}_2 & \mathbb{F}_2 & \mathbb{F}_2^2 & \mathbb{F}_2^5 & \mathbb{F}_2^6 & \mathbb{F}_2^8 & \dots & \\
 & \mathbb{F}_2 & \mathbb{F}_2 & \mathbb{F}_2^2 & \mathbb{F}_2^5 & \mathbb{F}_2^6 & \mathbb{F}_2^8 & \dots & \\
 & 0 & 0 & 0 & 0 & 0 & 0 & \dots & \\
 & \mathbb{F}_2 & \mathbb{F}_2 & \mathbb{F}_2^2 & \mathbb{F}_2^5 & \mathbb{F}_2^6 & \mathbb{F}_2^8 & \dots & \\
 \quad \strut & & & & & & & & \strut \\};
 \draw[-stealth] (m-6-2.south east) -- (m-8-5.north west);
 \draw[-stealth] (m-6-3.south east) -- (m-8-6.north west);
 \draw[-stealth] (m-6-4.south east) -- (m-8-7.north west);
 \draw[-stealth] (m-5-2.south east) -- (m-7-5.north west);
 \draw[-stealth] (m-5-3.south east) -- (m-7-6.north west);
 \draw[-stealth] (m-5-4.south east) -- (m-7-7.north west);
 \draw[-stealth] (m-4-2.south east) -- (m-6-5.north west);
 \draw[-stealth] (m-4-3.south east) -- (m-6-6.north west);
 \draw[-stealth] (m-4-4.south east) -- (m-6-7.north west);

 \draw[-stealth] (m-3-2.south east) -- (m-5-5.north west);
 \draw[-stealth] (m-3-3.south east) -- (m-5-6.north west);
 \draw[-stealth] (m-3-4.south east) -- (m-5-7.north west);
 \draw[thick] (m-1-1.east) -- (m-9-1.east) ;
 \draw[thick] (m-9-1.north) -- (m-9-9.north) ;
\end{tikzpicture}
    \caption{$E_3$-page of the spectral sequence associated to \ref{ss5}.}
    \label{fig:spectral53}
\end{figure}

\begin{figure}[H]
    \centering
    \begin{tikzpicture}
\matrix (m) [matrix of math nodes, 
 nodes in empty cells, nodes={minimum width=5ex, 
 minimum height=2ex, outer sep=-5pt}, 
 column sep=1ex, row sep=1ex]{ \quad\strut
 & & & & & & & & \\
 & \cdots & \cdots & \cdots & \cdots & \cdots & \cdots & \cdots & \\
 & 0 & 0 & 0 & 0 & 0 & 0 & \dots & \\
 & \mathbb{F}_2 & \mathbb{F}_2 & \mathbb{F}_2^2 & \mathbb{F}_2^4 & \mathbb{F}_2^5 & \mathbb{F}_2^6 & \dots & \\
 & \mathbb{F}_2 & \mathbb{F}_2 & \mathbb{F}_2^2 & \mathbb{F}_2^4 & \mathbb{F}_2^5 & \mathbb{F}_2^6 & \dots & \\
 & 0 & 0 & 0 & 0 & 0 & 0 & \dots & \\
 & 0 & 0 & 0 & 0 & 0 & 0 & \dots & \\
 & \mathbb{F}_2 & \mathbb{F}_2 & \mathbb{F}_2^2 & \mathbb{F}_2^5 & \mathbb{F}_2^6 & \mathbb{F}_2^8 & \dots & \\
 \quad \strut & & & & & & & & \strut \\};
 \draw[-stealth] (m-5-2.south east) -- (m-8-6.north west);
 \draw[-stealth] (m-5-3.south east) -- (m-8-7.north west);
 \draw[-stealth] (m-4-2.south east) -- (m-7-6.north west);
 \draw[-stealth] (m-4-3.south east) -- (m-7-7.north west);
 \draw[thick] (m-1-1.east) -- (m-9-1.east) ;
 \draw[thick] (m-9-1.north) -- (m-9-9.north) ;
\end{tikzpicture}
    \caption{$E_4$-page of the spectral sequence associated to \ref{ss5}.}
    \label{fig:spectral54}
\end{figure}

\begin{figure}[H]
    \centering
    \begin{tikzpicture}
\matrix (m) [matrix of math nodes, 
 nodes in empty cells, nodes={minimum width=5ex, 
 minimum height=2ex, outer sep=-5pt}, 
 column sep=1ex, row sep=1ex]{ \quad\strut
 & & & & & & & & \\
 & \cdots & \cdots & \cdots & \cdots & \cdots & \cdots & \cdots & \\
 & \mathbb{F}_2 & \mathbb{F}_2 & \mathbb{F}_2^2 & \mathbb{F}_2^4 & \mathbb{F}_2^4 & \mathbb{F}_2^5 & \dots & \\
 & 0 & 0 & 0 & 0 & 0 & 0 & \dots & \\
 & 0 & 0 & 0 & 0 & 0 & 0 & \dots & \\
 & 0 & 0 & 0 & 0 & 0 & 0 & \dots & \\
 & \mathbb{F}_2 & \mathbb{F}_2 & \mathbb{F}_2^2 & \mathbb{F}_2^4 & \mathbb{F}_2^4 & \mathbb{F}_2^5 & \dots & \\
 \quad \strut & & & & & & & & \strut \\};
 \draw[-stealth] (m-3-2.south east) -- (m-7-7.north west);
 \draw[thick] (m-1-1.east) -- (m-8-1.east) ;
 \draw[thick] (m-8-1.north) -- (m-8-9.north) ;
\end{tikzpicture}
    \caption{$E_5$-page of the spectral sequence associated to \ref{ss5}.}
    \label{fig:espectral55}
\end{figure}

\begin{figure}[H]
    \centering
    \begin{tikzpicture}
\matrix (m) [matrix of math nodes, 
 nodes in empty cells, nodes={minimum width=5ex, 
 minimum height=2ex, outer sep=-5pt}, 
 column sep=1ex, row sep=1ex]{ \quad\strut
 & & & & & & & & \\
 & \dots & \dots & \dots & \dots & \dots & \dots & \\
 & 0 & 0 & 0 & 0 & 0 & 0 & \dots & \\
 & 0 & 0 & 0 & 0 & 0 & 0 & \dots & \\
 & \mathbb{F}_2 & \mathbb{F}_2 & \mathbb{F}_2^2 & \mathbb{F}_2^4 & \mathbb{F}_2^4 & \mathbb{F}_2^4 & \dots & \\
 \quad \strut & & & & & & & & \strut \\};
 \draw[thick] (m-1-1.east) -- (m-6-1.east) ;
 \draw[thick] (m-6-1.north) -- (m-6-9.north) ;
\end{tikzpicture}
    \caption{$E_6$-page of the spectral sequence associated to \ref{ss5}.}
    \label{fig:espectral56}
\end{figure}

The first $ 6 $ columns of the $ 6^{th} $ page survive to the $ E_\infty $-page. Hence, they determine the cohomology groups of $ \overline{\Fl}_5(\mathbb{R}) $ up to degree $ 5 $. $ \overline{\Fl}_5(\mathbb{R}) $ being a $ 10 $-dimensional manifold, Poincar\'e duality determines the remaining ones.
In conclusion, we have proved the following result.
\begin{theorem} \label{poincare5}
The Poincar\'e series of the mod $ 2 $ cohomology ring of $ \overline{\Fl}_5(\mathbb{R}) $ is
\[
\Pi_{\overline{\Fl}_5(\mathbb{R})}(t) = t^{10} + t^9 + 2t^8 + 4t^7 + 4t^6 + 4t^5 + 4t^4 + 4t^3 + 2t^2 + 2 + 1.
\]
\end{theorem}

\section{The cohomology of \texorpdfstring{$\overline{\Fl}_3 (\mathbb{C})$}{text}}\label{sec:complex}

In this section, we provide a complete description of the cohomology of $\overline{\Fl}_3(\mathbb{C}) = U(3)/N(T_3) $. We denote by $\mathbb{Z}/p$ the ring of integers modulo $p$ and by $\mathbb{Z}_p$ the ring of $p$-adic integers. To compute this cohomology we use the Serre spectral sequence associated to the fiber sequence 
\begin{equation} \label{ssorder3}
 \Fl_3(\mathbb{C}) \to \overline{\Fl}_3(\mathbb{C}) \to B\altmathcal{S}_3.
\end{equation} 
The $E_2$ page of this spectral sequence is given by $E_2^{p, q} = H^p (B\altmathcal{S}_3 ; H^q (\Fl_3(\mathbb{C}))$ and the spectral sequence converges to $H^{*} (\overline{\Fl}_3(\mathbb{C}))$. By Theorem \ref{cohoflagcom}, $H^* (\Fl_3 (\mathbb{C}); \mathbb{Z} )$ is torsion-free, is of total rank $6$, with even dimensional cohomology concentrated in degrees $0, 2, 4, 6$. To determine $E_2^{p, q}$, we describe $H^d (\Fl_3 (\mathbb{C}) ; \mathbb{Z})$ as integral representations $M_d$ of $\altmathcal{S}_3$ for $d=0, 2, 4, 6$ and compute the group cohomology $H^* (\altmathcal{S}_3; M_d)$. The action of $\altmathcal{S}_3$ on $H^d (\Fl_3 (\mathbb{C}); \mathbb{Z})$ is via
\begin{itemize}
 \item the trivial representation $M_0$ in degree $d=0$, 
 \item the sign representation $M_6$ in degree $d=6$, 
 \item the standard representation $M_2$ and $M_4$ in degrees $d=2$ and $d=4$ respectively.
\end{itemize}
Also, note that $\{ x_1^{i_1} \cdot x_2^{i_2} \cdot x_3^{i_3} | i_j \le 3-j \} $ forms a basis for $H^* (\Fl_3 (\mathbb{C}) ; \mathbb{Z})$ and hence we can write
\begin{align*}
 H^0 (\Fl_3 (\mathbb{C}) ; \mathbb{Z}) &= \mathbb{Z} =: M_0 \\
 H^2 (\Fl_3 (\mathbb{C}) ; \mathbb{Z}) &= \frac{ \mathbb{Z} \{ x_1 , x_2 , x_3 \} }{(x_1+x_2+x_3)} =: M_2 \\
 H^4 (\Fl_3 (\mathbb{C}) ; \mathbb{Z}) &= \frac{ \mathbb{Z} \{ x_1 x_2 , x_2 x_3 , x_3 x_1 \}}{(x_1 x_2 + x_2 x_3 +x_1 x_3)} =: M_4\\
 H^6 (\Fl_3 (\mathbb{C}) ; \mathbb{Z} ) &= \mathbb{Z} \{x_1^2 x_2 \} := M_6. 
\end{align*}
As representations of $\altmathcal{S}_3$, the modules $M_2$ and $ M_4$ are isomorphic as they are both quotients of the standard rank three permutation module by its rank one submodule of invariants. More precisely, the isomorphism $M_2 \cong M_4$ can be described via the short exact sequence \[ 0 \to \mathbb{Z} \to \mathbb{Z} [C_3] \to M_2 \to 0, \]
where $C_3 \cong A_3 \subset \altmathcal{S}_3$ is isomorphic to the alternating group.
Similarly, the sign representation $M_6$ can be described via the short exact sequence \[ 0 \to \mathbb{Z} \to \mathbb{Z} [C_2] \to M_6 \to 0, \] where $C_2 := \altmathcal{S}_3 / C_3$ is isomorphic to $\mathbb{Z}/2$. Further, recall that the cohomology of $\altmathcal{S}_3$ is given by \[ H^d (\altmathcal{S}_3 ; M_0) = \begin{cases} \mathbb{Z}/2 & d = 4k+2 \\ \mathbb{Z}/6 & d=4k \\ 0 & \text{otherwise} \end{cases} \]
From the long exact sequences in group cohomology associated to the above short exact sequences, we obtain \[ H^d (\altmathcal{S}_3; M_2)_3 = \begin{cases}
 \mathbb{Z}/3 & d=4k+3 \\ 0 & \text{otherwise}
\end{cases} \quad \quad H^d (\altmathcal{S}_3 ; M_6)_3 = \begin{cases}
 \mathbb{Z}/3 & d=4k+2 \\ 0 & \text{otherwise}
\end{cases} \]
Let us now consider the Serre spectral sequence with the $3$-adic coefficients $\mathbb{Z}_3$ associated to the fibration (\ref{ssorder3}). We know that $ E_2^{0, 0} \cong \mathbb{Z}_3$ and $E_2^{0, j} = H^j (\Fl_3 (\mathbb{C}) , \mathbb{Z}_3)^{\altmathcal{S}_3} = 0$ for all $j>0$. Moreover, for the rows in the spectral sequence, we have for all $i\ge 0$ \begin{align*}
 E_2^{4(i+1), 0} \cong \mathbb{Z}/3 , &\quad E_2^{4i+3, 2} \cong \mathbb{Z}/3 \\
 E_2^{4i+3, 4} \cong \mathbb{Z}/3, &\quad E_2^{4i+2, 6} \cong \mathbb{Z}/3.
\end{align*}
All the other terms on the $ E_2$ page are zero.

\begin{figure}
    \centering
    \begin{tikzpicture}
 \matrix (m) [matrix of math nodes, 
 nodes in empty cells, nodes={minimum width=5ex, 
 minimum height=5ex, outer sep=-5pt}, 
 column sep=.5ex, row sep=2ex]{
 & & & & \\
 6 & 0 & 0 & \mathbb{Z}/3 & 0 & 0 & 0 & \mathbb{Z}/3 & 0 & 0 & \cdots \\
 4 & 0 & 0 & 0 & \mathbb{Z}/3 & 0 & 0 & 0 & \mathbb{Z}/3 & 0 &\cdots \\
 2 & 0 & 0 & 0 & \mathbb{Z}/3 & 0 & 0 & 0 & \mathbb{Z}/3 & 0 &\cdots \\
 0 & \mathbb{Z}_3 & 0 & 0 & 0 & \mathbb{Z}/3 & 0 & 0 & 0 & \mathbb{Z}/3 &\cdots \\
 \quad\strut & 0 & 1 & 2 & 3 & 4 & 5 & 6 & 7 & 8 & \strut \\};
\draw[-stealth] (m-2-4.south east) -- (m-4-9.north west);
\draw[-stealth] (m-3-5.south east) -- (m-5-10.north west);

\draw[thick] (m-1-1.east) -- (m-6-1.east) ;
\draw[thick] (m-6-1.north) -- (m-6-11.north) ;
\end{tikzpicture}
    \caption{$E_2$-page of the spectral sequence associated to \ref{ssorder3} with $3$-adic coefficients.}
    \label{fig:ss33}
\end{figure}

We note that, among the differentials, $d_3 $ and $d_7$ must be zero as $E_3^{4i+5, 4} = E_3^{4i+6, 2} = E_3^{4i+6, 0} = 0$ and $E_7^{4i+9, 0} = 0$. Since  $\dim (\overline{\Fl}_3(\mathbb{C}))=6$, we must have $d_5 \neq 0$ and for all $i\ge 0$, and \begin{align*}
 d_5 : E_5^{4i+2, 6} \xrightarrow{\cong} E_5^{4(i+1)+3, 2} \\
 d_5 : E_5^{4i+3, 4} \xrightarrow{\cong} E_5^{4(i+1)+4, 0}
\end{align*}
Hence, the only remaining terms surviving to the $E_{\infty}$ page are $E_{\infty}^{3, 2} = E_{\infty}^{4, 0} \cong \mathbb{Z}/3$. Thus, from the above discussion, we obtain the following.
\begin{theorem}
 The $3$-adic cohomology of $ \overline{\Fl}_3 (\mathbb{C})$ is given by \[ H^d (\overline{\Fl}_3(\mathbb{C}); \mathbb{Z}_3) \cong \begin{cases}
 \mathbb{Z}_3 & d = 0 \\ \mathbb{Z}/3 & d=4, 5 \\ 0 & \text{otherwise}
 \end{cases} \]
\end{theorem}

We now focus on the mod 2 case. We have an isomorphism $H^* (\altmathcal{S}_3; \mathbb{Z})_2 \cong H^* (C_2 ; \mathbb{Z})_2$ induced by the inclusion $\mathbb{Z} \hookrightarrow \mathbb{Z} [\altmathcal{S}_3 / C_2]$. So, we see that $H^d (\altmathcal{S}_3; M_2)_2 = 0$ for all $i\ge 0$. From the long exact sequence in group cohomology associated to the short exact sequence \[ 0 \to \mathbb{Z} \to \mathbb{Z} [\altmathcal{S}_3/C_3] \to M_6 \to 0, \] 
we deduce that $H^d (\altmathcal{S}_3 ; M_6)_2 \cong H^{d+1} (\altmathcal{S}_3 ; \mathbb{Z})_2$. So, \[ H^d (\altmathcal{S}_3 ; M_6)_2 \cong \begin{cases}
 \mathbb{Z}/2 & d = 2k+1 \\
 0 & \text{otherwise}
\end{cases} \]
 As before, we consider the same Serre spectral sequence, but now with the $2$-adic coefficients $\mathbb{Z}_2$. We have that $E_2^{0, 0} \cong \mathbb{Z}_2$ and for all $i \ge 0$ \begin{align*}
 E_2^{2i+1, 6} \cong \mathbb{Z}/2, \quad E_2^{2i+2, 0} \cong \mathbb{Z}/2 
 \end{align*}
and all other terms on the $E_2$ page are zero. Looking at the possible non-zero differentials, the only possibility is \[ d_7: E_7^{2i+1, 6} \to E_7^{2i+8, 0} \] 
which must be an isomorphism for $i\ge 0$, for dimension reasons.

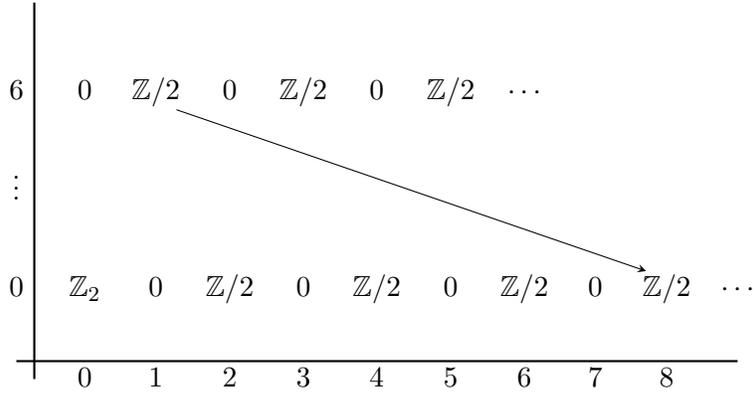
\begin{figure}
    \centering
    \begin{tikzpicture}
 \matrix (m) [matrix of math nodes, 
 nodes in empty cells, nodes={minimum width=5ex, 
 minimum height=5ex, outer sep=-5pt}, 
 column sep=.5ex, row sep=2ex]{
 & & & & \\
 6 & 0 & \mathbb{Z}/2 & 0 & \mathbb{Z}/2 & 0 & \mathbb{Z}/2 & \cdots \\
 \vdots& & & & & & & \\
 0 & \mathbb{Z}_2 & 0 & \mathbb{Z}/2 & 0 & \mathbb{Z}/2 & 0 & \mathbb{Z}/2 & 0 & \mathbb{Z}/2 & \cdots \\
 \quad\strut & 0 & 1 & 2 & 3 & 4 & 5 & 6 & 7 & 8 & \strut \\};
\draw[-stealth] (m-2-3.south east) -- (m-4-10.north west);

\draw[thick] (m-1-1.east) -- (m-5-1.east) ;
\draw[thick] (m-5-1.north) -- (m-5-11.north) ;
\end{tikzpicture}
    \caption{$E_2$-page of the spectral sequence associated to \ref{ssorder3} with $2$-adic coefficients.}
    \label{fig:ss32}
\end{figure}

Hence the only surviving terms on the $E_{\infty}$ page are $E_{\infty}^{2, 0} = E_{\infty}^{4, 0} =E_{\infty}^{6, 0} \cong \mathbb{Z}/2$. Therefore, we obtain the following theorem.
\begin{theorem} \label{coh3com2}
 The $2$-adic cohomology of $\overline{\Fl}_3 (\mathbb{C})$ is given by \[ H^d (\overline{\Fl}_3(\mathbb{C}); \mathbb{Z}_2) \cong \begin{cases}
 \mathbb{Z}_2 & d = 0 \\ \mathbb{Z}/2 & d=2, 4, 6 \\ 0 & \text{otherwise}
 \end{cases} \]
\end{theorem}
As a corollary of the above theorem, we can have the following.
\begin{corollary} \label{cohoring3}
 Let $\gamma$ denote the non-zero class in $H^1 (\overline{\Fl}_3(\mathbb{C}); \mathbb{F}_2) \cong \mathbb{F}_2$. Then \[ H^* (\overline{\Fl}_3(\mathbb{C}); \mathbb{F}_2) \cong \frac{\mathbb{F}_2 [\gamma]}{(\gamma^7)} . \]
\end{corollary}

\section{Auerbach bases} \label{auerbach}

In \cite{WW17}, Weber--Wojciechowski produced the following estimate for the number of Auerbach bases (see the definition in \S \ref{intro}) of a finite-dimensional Banach space and thus settled the Pełczyński conjecture.
\begin{theorem} \label{weberestimate}
 1) In every $n$-dimensional Banach space there exist at least $\frac{n(n-1)}{2} + 1$
different Auerbach bases. 
\newline
2) For an open (in the sense of Hausdorff distance) dense set of Banach spaces with
smooth (continually twice differentiable) norm, there exists at least 
$2^{[n/2]} + 4$ (real case) or
$n!$ (complex case) different Auerbach bases.
\end{theorem}

Below we provide a brief account of their method of proof. We first need the following definition. 
 
\begin{definition} \label{bifurcation} Let $M$ be a topological manifold. For a continuous function $g: M \rightarrow \mathbb{R}$, a point $x \in M$ is a topologically \textit{regular point} of $g$ if there exists a neighborhood $U$ of $x$ and homeomorphism 
$h: S \times (a -\epsilon, a + \epsilon) \rightarrow U $ for some $S$ such that $g \circ h$ is the projection onto the second factor. If $x$ is not topologically regular, then it is called a \textit{bifurcation point}. 
\end{definition}
$S_{\mathbb{X}}^n$, the $n$-fold cartesian product of $S_{\mathbb{X}}$, may be viewed as matrices with rows unit vectors. So the determinant $\det: S_{\mathbb{X}}^n \rightarrow \mathbb{R}$ is well defined.

For a basis $W=\{v_1, \cdots, v_n\}$ of $\mathbb{X}$, for $1 \leq i \leq n$, let \[ |\mathrm{det}|^i_{W}:S_{\mathbb{X}} \to \mathbb{R}, \quad |\mathrm{det}|^i_{W}(x)= |\det (v_1, v_2, \cdots, v_{i-1}, x, v_{i+1}, \cdots, v_n)| \] denote the restriction map. In the complex case, $S_{\mathbb{X}}$ may be viewed as a subspace of $\mathbb{R}^{2n}$. Then, we have the following characterization of Auerbach bases. 

\begin{theorem} \label{int2auerbach}
Let $W= \{v_1, \cdots, v_n\}$ be a basis of a Banach space $\mathbb{X}$ with $v_i \in S_{\mathbb{X}} $ for $i=1, 2, \cdots, n$. Then, $W$ is an Auerbach basis of $\mathbb{X}$ if and only if $v_i$ is a bifurcation point of $|\mathrm{det}|_W^i$ for all $i$. 
 
\end{theorem}

\begin{proof}

Let $\{v_1, v_2, \cdots v_n\}$ be an Auerbach basis of a Banach space $\mathbb{X} $ over $\mathbbm{k}= \mathbb{R}$ or $ \mathbb{C}$. Then for each $v_i$, 

\begin{equation}\label{birkhoffjames}
| v_i + \lambda v|| \geq ||v_i||= 1 \ \text{for all} \ \lambda \in \mathbbm{k}, 
\end{equation}
for all $v \in \mathrm{span} \{ v_1, v_2, \cdots, v_{i-1}, v_{i+1}, \cdots, v_n\}$.

We will first prove that $v_n$ is a bifurcation of $|\mathrm{det}|_W^n$. 
For $x \in S_{\mathbb{X}}$, we can write $x= r v_n + \bar{v}$ for some $\bar{v} \in \mathrm{span} \{ v_1, v_2, \cdots, v_{n-1}\}$. Then, by Equation \ref{birkhoffjames}, $|r| \leq 1$; otherwise, plugging in $ \lambda= -\frac{1}{r}$ and $v= \bar{v}$ into the equation, one gets $|| v_n - v_n + \frac{1}{r} x||= |1/ r| < 1$, a contradiction.
\\
This further implies 
\[ || x -\text{proj}_n(x)|| \leq || v_n -\text{proj}_n(v_n)|| \] 
where $\text{proj}_n: \mathbb{X} \to \mathrm{span} \{ v_1, v_2, \cdots, v_{n-1}\}$ denote the standard projection. 

For $v \in \mathbb{X}$, let $\{u_1, u_2, \cdots, u_n\}$ be the basis obtained from $\{ v_1, v_2, \cdots, v_{n-1}, v \}$ by the Gram-Schmidt orthogonalization process, then we know that
\begin{equation*} 
|\det( v_1, v_2, \cdots, v_{n-1}, v)|= \prod_{i=1}^{n} ||u_i||= \prod_{i=1 }^{n-1} ||u_i|| \cdot || v -\text{proj}_n(v)||.
\end{equation*} 
Hence, from the above equations, $\forall x \in S_{\mathbb{X}}$, 
 \[ |\mathrm{det}|_W^n (x)
\leq |\mathrm{det}|_W^n (v_n). \]
 
Therefore, $v_n$ is a bifurcation point of $|\det|^n_W$. By a similar argument, $v_j$ is a bifurcation point of $|\det|^j_W$ for all $j$.
\newline
For the "if" part, suppose that 
 $\{ v_1, v_2, \cdots, v_{n-1}, v_{n}\}$ does not form an Auerbach basis. Without loss of generality, we may assume that $\mathrm{span} \{ v_1, v_2, \cdots, v_{n-1}\} \cap {B}_{\mathbb{X}} \neq \emptyset$, where ${B}_{\mathbb{X}}$ denote the open unit ball in $\mathbb{X}$. Then, we need to show that there is a neighborhood $N_{v_n} $ of $v_n$ in $S_{\mathbb{X}}$ and a topological space $S$ and a homeomorphism 
 \[ h: S \times (a - \epsilon, a+ \epsilon) \rightarrow N_{v_n}, \]
 where $a= |\det (v_1, v_2, \cdots, v_n)| \neq 0$ such that $|\det|^n_W \circ h$ is the projection on to the second factor. 
\newline
By continuity, there exists a neighborhood $N_{v_n} \subset {S}_{\mathbb{X}} $ of $v_n$ such that 
\[ x+ \text{span} \{ v_1, v_2, \cdots, v_{n-1}\} \cap \ {B}_{\mathbb{X}} \neq \emptyset \quad \text{for all} \ x \in N_{v_n} .\]
By convexity of $\bar{B}_{\mathbb{X}}= {B}_{\mathbb{X}} \cup {S}_{\mathbb{X}}$, for all $ x \in N_{v_n}$, we have 

\[
x+ \text{span} \{ v_1, v_2, \cdots, v_{n-1}\} \cap N_{v_n} \cong 
\begin{cases}
 D^{n-1} & \text{in the real case} \\
 D^{2n-3} & \text{in the complex case}
\end{cases}
\]
Let $\epsilon >0$ be such that 
\[ (|\mathrm{det}|_W^j)^{-1} (r) \cap N_{v_n} \neq \emptyset \quad \text{for all} \; r \in (a- \epsilon, a+ \epsilon) \]
 
Then, in the real case, $ \forall r \in (a- \epsilon, a+ \epsilon)$, there is a $ x \in N_{v_n} $ such that

\begin{align*} (|\mathrm{det}|_W^n)^{-1} (r) \cap N_{v_n}= (x+ \text{span} \{ v_1, v_2, \cdots, v_{n-1}\}) \cap N_{v_n} \cong D^{n-2} 
\end{align*}
Therefore, there is a homeomorphism 
\[ h: D^{n-2} \times (a- \epsilon, a+ \epsilon) \rightarrow N_{v_n} \] 
such that $|\mathrm{det}|_W^n \circ h $ is projection on to the last factor.
\newline
In the complex case, $ \forall r \in (a- \epsilon, a+ \epsilon)$, there is a $ x \in N_{v_n} $ such that
\begin{align*} (|\mathrm{det}|_W^j)^{-1} (r) \cap N_{v_n}=(e^{i \theta}( x+ \text{span} \{ v_1, v_2, \cdots, v_{n-1}\})) \cap N_{v_n} \cong S^1 \times D^{2n-3} \quad \text{for} \; \theta \in [0, 2\pi). 
\end{align*}
 
Therefore, the map 
\[ h: S^1 \times D^{2n-3} \times (a- \epsilon, a+ \epsilon) \rightarrow N_{v_n} \] 
is our required homeomorphism.
\end{proof}

Clearly, if a basis $W=\{v_1, \cdots, v_n\}$ is a bifurcation point of $|det|$, then for each $i$, $v_i$ is a bifurcation point of $|\mathrm{det}|^i_W$. Furthermore, $|det|$ descends to a well defined map on the quotient space  $\mathbb{S}_{\mathbb{X}, \neq 0}/ \Bgroup_n $ in the real case and $\mathbb{S}_{\mathbb{X}, \neq 0}/ N (T_n)$ in the complex case.  So, we have the following corollary.

\begin{corollary} \label{detandab} The bifurcation points of $|det|: S_{\mathbb{X}, \neq 0}^n \rightarrow \mathbb{R}$ correspond to the Auerbach bases of $\mathbb{X}$. Furthermore, two distinct bifurcation points of the induced map on $\mathbb{S}_{\mathbb{X}, \neq 0}/ \Bgroup_n $ ( $\mathbb{S}_{\mathbb{X}, \neq 0}/ N (T_n)$ in the complex case) correspond to two distinct equivalence classes of Auerbach bases.
\end{corollary}

For a Banach space with $C^1$ norm, that is, when $\mathbb{S}_{\mathbb{X}}$ is a $C^1$-manifold, we have a one-to-one correspondence between the critical points of the determinant and the Auerbach bases; see \cite[Proposition~2.1]{MS23} for an independent proof of this fact.
\newline
 The Gram-Schmidt process commutes with the action of $\Bgrouppos_n$ and $N_{U(n)}(T_n)$ in the real and complex case, respectively. So, we have the following; see also \cite[Proposition 16]{WW17}.  
\begin{proposition}\label{transition}
There are homeomorphisms
\begin{align*}
    \mathbb{S}_{\mathbb{X}, \neq 0}/ \Bgroup_n \cong \overline{\Fl}_n(\mathbb{R}) \times D^{ \frac{n(n-1)}{2}} \quad \text{in the real case}, \tag*{and} \\
 \mathbb{S}_{\mathbb{X}, \neq 0}/ N (T_n) \cong \overline{\Fl}_n(\mathbb{C}) \times D^{ n(n-1)} \quad \text{in the complex case}. 
\end{align*}

\end{proposition}
The following is well known; see \cite{Mi}. 
\begin{theorem}
Let $M$ be a path-connected compact manifold. Let
$f : M \to \mathbb{R}$ be a smooth function. The number of critical points of $f$ is not smaller than $\mathrm{rank}(H^*(M))$.
\end{theorem}
The \textit{Lusternik-Schnirelmann} category $\mathrm{cat}(X)$ of a space $X$ denotes the smallest cardinality of covering of $X$ by open sets which are contractible in $X$. The following theorem is proved in \cite[Theorem A.4]{WW17}. 

\begin{theorem}
Let $M$ be a path-connected metric space that is locally contractible. Let
$f : M \to \mathbb{R}$ be a continuous function. Then the number of bifurcation points of $f$ is not smaller than $\mathrm{cat}(M)$.
\end{theorem}

Let $R$ be a ring. The cup-length $ \ell_{R}(M )$
is the length of the longest sequence of $\alpha_1, \alpha_2, \dots , \alpha_l$ of cohomology classes $\alpha_j \in H^*(M, R)$
of positive degree such that
\[ \alpha_1 \alpha_2 \cdots \alpha_l \neq 0. \]
The following are well known; see \cite{COT03}. 
\begin{theorem}\label{catbound} Let $M$ be a topological space. Then for any ring $R$
\[ \dim(M) +1 \geq \mathrm{cat}(M ) \geq \ell_{R}(M ) +1. \]
\end{theorem}

Let $f: M_1 \to M_2$ be a finite covering of topological spaces, then 
\[ \mathrm{cat}(M_2) \geq \mathrm{cat}(M_1). \]

Note that ${\Fl}_n (\mathbb{R}) \rightarrow \overline{\Fl}_n (\mathbb{R})$ is a finite covering. Thus, from Theorem \ref{cohoflagreal} and \ref{catbound} and the above discussion
\[ \mathrm{cat}\big(\overline{\Fl}_n(\mathbb{R}) \big)= \frac{n(n-1)}{2} +1= \dim (\overline{\Fl}_n(\mathbb{R})) . \]
So, we have the optimum estimate for $\mathrm{cat}\big(\overline{\Fl}_n(\mathbb{R})$. 

From Proposition \ref{rationalcoho} and Proposition 3D.4 in \cite{AT}, one obtains that 
\[\text{rank}(H^*(\overline{\Fl}_n(\mathbb{R}); \mathbb{Q})= 2^{[n/2]}.\]
Then a proof of the second part of Theorem \ref{weberestimate}, in the real case, essentially follows from the above discussion.
\newline 
In the complex case, from Theorem \ref{cohoflagcom}, \ref{catbound} and the above discussion, we have

\begin{equation}\label{catgap}
\dim (\overline{\Fl}_n(\mathbb{C}))= n(n-1) \geq  
\mathrm{cat} \big(\overline{\Fl}_n(\mathbb{C}) \big) \geq \frac{n(n-1)}{2} +1.
\end{equation}
The first part of Theorem \ref{weberestimate}, in the complex case, follows from this.  
From our computational result in Corollary \ref{cohoring3} and Theorem \ref{catbound}, we obtain 
\begin{equation}\label{cat7}
\mathrm{cat}\big(\overline{\Fl}_3(\mathbb{C})) \geq \ell_{\mathbb{F}_2}(\overline{\Fl}_3(\mathbb{C})) +1 = 7= \dim (\overline{\Fl}_3(\mathbb{C})) +1 .
\end{equation}
This implies 
\[ \mathrm{cat}\big(\overline{\Fl}_3(\mathbb{C})) = \dim (\overline{\Fl}_3(\mathbb{C})) +1 \]
So, in this case, the maximum possible value for the category specified by Equation \ref{catgap} has been achieved. 

The proof of the complex part of the second statement of Theorem \ref{weberestimate} was achieved in \cite[Corollary 18-19]{WW17}, using a similar argument as in the real case, specifically through estimating the rational cohomology of the finite covering ${\Fl}_n(\mathbb{C})$ of $\overline{\Fl}_n(\mathbb{C})$. However, according to Theorem \ref{rationalcoho}, this argument is not applicable. Instead, in this case, the mod-2 cohomology of $\overline{\Fl}_n(\mathbb{C})$ provides a more accurate estimate.

To summarize, we have the following improvements of Theorem \ref{weberestimate} for lower dimensions. 
\begin{corollary} \label{improvedestimate}
 1) In every $3$-dimensional complex Banach space there exists at least $7$ different Auerbach bases. 
\newline
2) For an open (in the sense of Hausdorff distance) dense set of Banach spaces of dimension $4$ and $5$ with
smooth (continually twice differentiable) norm, there exist at least $12$ and
$28$ different Auerbach bases, respectively.
\end{corollary}
 It is known that (cf. \cite{sob64}) if $K$ is an $n$-dimensional convex body then any Auerbach basis of a Banach space with the unit ball $K - K= \{x - y| x, y \in K\}$ corresponds to the system of conjugate affine diameters of $K$. Thus, the above corollary also improves the estimate for the number of different systems of conjugate affine diameters in the respective cases. 

Thus, in conclusion, one can expect that the estimate for Auerbach bases given by Theorem \ref{weberestimate} can be much improved in higher dimensions as well.

\section*{Acknowledgement} The first author acknowledges the MIUR Excellence Department Project awarded to the Department of Mathematics, University of Rome Tor Vergata, CUP E83C18000100006. The third author was partially supported by Science and Engineering Research Board (SERB) grant (CRG/2019/007028) for the year 2020-2021. 

\bibliographystyle{amsplain}
\bibliography{refile}

\providecommand{\bysame}{\leavevmode\hbox to3em{\hrulefill}\thinspace}
\providecommand{\MR}{\relax\ifhmode\unskip\space\fi MR }
\providecommand{\MRhref}[2]{%
  \href{http://www.ams.org/mathscinet-getitem?mr=#1}{#2}
}
\providecommand{\href}[2]{#2}
\begin{thebibliography}{10}

\bibitem{AGV20}
Omar Antol\'{\i}n-Camarena, Simon~Philipp Gritschacher, and Bernardo
  Villarreal, \emph{Classifying spaces for commutativity of low-dimensional
  {L}ie groups}, Math. Proc. Cambridge Philos. Soc. \textbf{169} (2020), no.~3,
  433--478. \MR{4170613}

\bibitem{Brenti-Reiner-Roichman}
Francesco Brenti, Victor Reiner, and Yuval Roichman, \emph{Alternating
  subgroups of {C}oxeter groups}, J. Combin. Theory Ser. A \textbf{115} (2008),
  no.~5, 845--877. \MR{2417024}

\bibitem{Brown}
Edgar~H. Brown, Jr., \emph{The cohomology of {$B{\rm SO}_{n}$} and {$B{\rm
  O}_{n}$} with integer coefficients}, Proc. Amer. Math. Soc. \textbf{85}
  (1982), no.~2, 283--288. \MR{652459}

\bibitem{COT03}
Octav Cornea, Gregory Lupton, John Oprea, and Daniel Tanr\'{e},
  \emph{Lusternik-{S}chnirelmann category}, Mathematical Surveys and
  Monographs, vol. 103, American Mathematical Society, Providence, RI, 2003.
  \MR{1990857}

\bibitem{Salvetti:00}
C.~De~Concini and M.~Salvetti, \emph{Cohomology of {C}oxeter groups and {A}rtin
  groups}, Math. Res. Lett. \textbf{7} (2000), no.~2-3, 213--232. \MR{1764318}

\bibitem{FMS22}
Greg Friedman, Anibal Medina, and Dev Sinha, \emph{Foundations of geometric
  cohomology: from co-orientations to product structures}, available on
  arXiv.org, 2022.

\bibitem{fulton1996}
William Fulton, \emph{Geometry}, London Mathematical Society Student Texts,
  p.~127–130, Cambridge University Press, 1996.

\bibitem{Fulton-Pragracz}
William Fulton and Piotr Pragacz, \emph{Schubert varieties and degeneracy
  loci}, Lecture Notes in Mathematics, vol. 1689, Springer-Verlag, Berlin,
  1998, Appendix J by the authors in collaboration with I. Ciocan-Fontanine.
  \MR{1639468}

\bibitem{Giusti-Salvatore-Sinha}
Chad Giusti, Paolo Salvatore, and Dev Sinha, \emph{The mod-2 cohomology rings
  of symmetric groups}, J. Topol. \textbf{5} (2012), no.~1, 169--198.
  \MR{2897052}

\bibitem{Giusti-Sinha}
Chad Giusti and Dev Sinha, \emph{Mod-two cohomology rings of alternating
  groups}, J. Reine Angew. Math. \textbf{772} (2021), 1--51. \MR{4227589}

\bibitem{LG21}
Lorenzo Guerra, \emph{The {$\rm mod\, 2$} cohomology of the infinite families
  of {C}oxeter groups of type {$B$} and {$D$} as almost-{H}opf rings}, Algebr.
  Geom. Topol. \textbf{23} (2023), no.~7, 3221--3292. \MR{4647676}

\bibitem{Guerra-Santanil}
Lorenzo {Guerra} and Santanil {Jana}, \emph{{Cohomology of complete unordered
  flag manifolds}}, arXiv e-prints (2023), arXiv:2309.00429.

\bibitem{AT}
Allen Hatcher, \emph{Algebraic topology}, Cambridge University Press,
  Cambridge, 2002. \MR{1867354}

\bibitem{Humphreys}
James~E. Humphreys, \emph{Reflection groups and {C}oxeter groups}, Cambridge
  Studies in Advanced Mathematics, vol.~29, Cambridge University Press,
  Cambridge, 1990. \MR{1066460}

\bibitem{MS23}
Arun {Maiti} and Debmalya {Sain}, \emph{{On the Auerbach bases of $l^{n}_p$
  spaces}}, arXiv e-prints (2023), arXiv:2302.00664.

\bibitem{JM01}
John McCleary, \emph{A user's guide to spectral sequences}, second ed.,
  Cambridge Studies in Advanced Mathematics, vol.~58, Cambridge University
  Press, Cambridge, 2001. \MR{1793722}

\bibitem{Mi}
John Milnor, \emph{Morse theory}, Based on lecture notes by M. Spivak and R.
  Wells. Annals of Mathematics Studies, No. 51, Princeton University Press,
  Princeton, N.J., 1963. \MR{0163331}

\bibitem{sob64}
Andrew Sobczyk, \emph{Convex polygons}, Proc. Amer. Math. Soc. \textbf{15}
  (1964), 438--446. \MR{161221}

\bibitem{WW17}
Andrzej Weber and Micha\l~\ Wojciechowski, \emph{On the {P}e\l czy\'{n}ski
  conjecture on {A}uerbach bases}, Commun. Contemp. Math. \textbf{19} (2017),
  no.~6, 1750016, 19. \MR{3691505}

\end{thebibliography}

\appendix
\section{Appendix} \label{appendix} 
For the reader's convenience, here we provide the details of the computations of $H^*(\Bgrouppos_n)$ given in Table \ref{basis table}. We recall from Lemma \ref{connectinghom}, the restriction $ \rho_n \colon H^*(\Bgroup_n) \to H^*(\Bgrouppos_n) $ induces an injective map
\[
\overline{\rho}_n \colon \frac{H^*(\Bgroup_n)}{(\gamma_{1, 1} \odot 1_1 + \alpha_1 \odot 1_2)} \to H^*(\Bgrouppos_n), 
\]
with cokernel given by the annihilator ideal of $ \gamma_{1, 1} \odot 1_1 + \alpha_1 \odot 1_2 $.

In light of this, in order to describe cohomologies of $\Bgrouppos_3$, we first compute the additive generators of the ideal $ \ker(\rho_3) $. Let $x=\gamma_{1, 1} \odot 1_{1} + \alpha_1 \odot 1_{2}$. We compute below $x \cdot y$, for all possible elements of the basis of Hopf monomials (see \ref{basis B}) $y$. 
\newline
For $y=(\alpha_1^l\odot \alpha^m_2)$, 
\[ x\cdot (\alpha^l_1 \odot \alpha_2^m)= \alpha_1^l \odot \gamma_{1, 1} \alpha_2^m + (\alpha_1 \odot 1_2)\cdot ( \alpha^l_1 \odot \alpha_2^m) . \]
Using Theorem~\ref{cohob}, 
\begin{gather*} \Delta (\alpha_1 \odot 1_2)= \Delta \alpha_1 \odot \Delta 1_2= (\alpha_1 \otimes 1_0 + 1_0 \otimes \alpha_1) \odot ( 1_2 \otimes 1_0 +1_1 \otimes 1_1 +1_0 \otimes 1_2 ) \\
=(\alpha_1 \odot 1_2) \otimes 1_0 + (\alpha_1 \odot 1_1) \otimes 1_1 + \alpha_1 \otimes 1_2 + 1_2 \otimes \alpha_1 + 1_1 \otimes (\alpha_1 \odot 1_1) + 1_0 \otimes (\alpha_1 \odot 1_2) .
\end{gather*}
Using the Hopf ring distributivity and the fact that the cup product of terms in different components is zero, the second summand in the above expression becomes
\[ (\alpha_1 \odot 1_2)\cdot(\alpha^l_1 \odot \alpha_2^m)= 
\alpha^{l+1}_1 \odot \alpha_2^m + \alpha^l_1 \odot \alpha^{m+1}_1 \odot \alpha^{m}_1 . \]

Similarly
\begin{gather*}
\Delta(1_1 \odot \gamma_{1, 1}) = 1_0 \otimes (1_1 \odot \gamma_{1, 1}) + 1_1 \otimes \gamma_{1, 1} + \gamma_{1, 1} \otimes 1_1 + (1_1 \odot \gamma_{1, 1}) \otimes 1_0 \\ \mbox{ and} \quad
(1_1 \odot \gamma_{1, 1}) \cdot y = \alpha_1^l \odot \gamma_{1, 1} \alpha_2^m .
\end{gather*}

Thus
\[ x\cdot ( \alpha_1^l \odot \alpha^m_2)= \alpha_1^l \odot \gamma_{1, 1} \alpha_2^m + \alpha_1^{l+1} \odot \alpha_2^{m} + \alpha^l_1 \odot \alpha^{m+1}_1 \odot \alpha^{m}_1 . \]

Similarly, we obtain 
\[ x \cdot ( \alpha^l_1 \odot \gamma^m_{1, 1} \alpha_2^n)= \alpha^l_1 \odot \gamma^{m+1}_{1, 1} \alpha_2^n + \alpha^{l+1}_1 \odot \gamma^{m}_{1, 1} \alpha_2^n , \]

and
\[ x \cdot (\alpha_1^l \odot \alpha_1^m \odot \alpha_1^n) = \alpha_1^{l+1} \odot \alpha_1^m \odot \alpha_1^n + \alpha_1^l \odot \alpha_1^{m+1} \odot \alpha_1^n + \alpha_1^l \odot \alpha_1^m \odot \alpha_1^{n+1} \]

and lastly
\[ x \cdot \alpha^m_3= \gamma_{1, 1} \alpha^m_2 \odot \alpha_1^{m} + \alpha^{m+1}_1 \odot \alpha_2^m . \]

In order to describe cohomologies of $\Bgrouppos_4$, we first compute the additive generators of the ideal $\ker(\rho_4)$. Let $x=\gamma_{1, 1} \odot 1_{2} + \alpha_1 \odot 1_{3}$. We compute below $x \cdot y$, for all Hopf monomials $y$. The calculations are essentially the same as those for $ \ker(\rho_3) $.

We first compute the coproduct of $ x $ using the Hopf algebra structure:
\begin{gather*}
\Delta(x) = \Delta(\alpha_1) \odot \Delta(1_3) + \Delta(\gamma_{1, 1}) \odot \Delta(1_2) = (\alpha_1 \odot 1_3) \otimes 1_0 + (\alpha_1 \odot 1_1) \otimes 1_2 + (\alpha_1 \odot 1_2) \otimes 1_1 \\
+ \alpha_1 \otimes 1_3 + 1_3 \otimes \alpha_1 + 1_1 \otimes (\alpha_1 \odot 1_2) + 1_2 \otimes (\alpha_1 \odot 1_1)+ 1_0 \otimes (\alpha_1 \odot 1_3) + 1_2 \otimes \gamma_{1, 1} \\
+ 1_0 \otimes (\gamma_{1, 1} \odot 1_2) + 1_1 \otimes (1_1 \odot \gamma_{1, 1}) + (1_1 \odot \gamma_{1, 1}) \otimes 1_1 + (\gamma_{1, 1} \odot 1_2) \otimes 1_0 + 1_2 \otimes \gamma_{1, 1} .
\end{gather*}

Using the coproduct computation above and Hopf ring distributivity we obtain the cup product of $ x $ and all possible Hopf monomials $ y $ in degree at most $ 3 $:
\begin{gather*}
x \cdot 1_4 = x, x \cdot (\gamma_{1, 1}^k \odot \gamma_{1, 1}^l) = \gamma_{1, 1}^{k+1} \odot \gamma_{1, 1}^l + \gamma_{1, 1}^k \odot \gamma_{1, 1}^{l+1} \mbox{ for } l < k, \\
x \cdot (\gamma_{1, 1}^k \odot \alpha_1^l \odot \alpha_1^m \odot 1_1) = \gamma_{1, 1}^{k+1} \odot \alpha_1^l \odot \alpha_1^m \odot 1_1 + \gamma_{1, 1}^k \odot \alpha_1^{l+1} \odot \alpha_1^m \odot 1_1 + \gamma_{1, 1}^k \odot \alpha_1^l \odot \alpha_1^{m+1} \odot 1_1 \\
+ \gamma_{1, 1}^k \odot \alpha_1^l \odot \alpha_1^m \odot \alpha_1 \mbox{ for } m < l, x \cdot \gamma_{1, 2} = \gamma_{1, 1}^2 \odot \gamma_{1, 1}, x \cdot (\alpha_2 \odot 1_2) = \alpha_1^2 \odot \alpha_1 \odot 1_2 + \alpha_3 \odot 1_1 \\
+ (\gamma_{1, 1} \alpha_2) \odot 1_2, x \cdot (\alpha_3 \odot 1_1) = \alpha_1^2 \odot \alpha_2 \odot 1_1 + (\gamma_{1, 1}\cdot \alpha_2) \odot \alpha_1 \odot 1_1, x \cdot \gamma_{2, 1} = 0.
\end{gather*}

Finally, in $ H^*(\Bgroup_5) $, the same calculations yield the following formulas for the cup product of $ x = \gamma_{1, 1} \odot 1_3 + \alpha_1 \odot 1_4 $ with Hopf monomials $ y $ of degree at most $ 5 $:
\begin{gather*}
x \cdot 1_5 = x, x \cdot (\gamma_{1, 1}^k \odot \gamma_{1, 1}^l \odot \alpha_1^m) = \gamma_{1, 1}^{k+1} \odot \gamma_{1, 1}^l \odot \alpha_1^m + \gamma_{1, 1}^k \odot \gamma_{1, 1}^{l+1} \odot \alpha_1^m + \gamma_{1, 1}^k \odot \gamma_{1, 1}^l \odot \alpha_1^{m+1} \mbox{ for } k > l, \\
x \cdot (\gamma_{1, 2}^k \odot \alpha_1^l) = \gamma_{1, 1}^{k+1} \odot \gamma_{1, 1}^k \odot \alpha_1^l + \gamma_{1, 2}^k \odot \alpha_1^{l+1}, x \cdot (\gamma_{1, 1}^k \odot \alpha_1^l \odot \alpha_1^m \odot \alpha_1^h) = \gamma_{1, 1}^{k+1} \odot \alpha_1^l \odot \alpha_1^m \odot \alpha_1^h \\
+ \gamma_{1, 1}^k \odot \alpha_1^{l+1} \odot \alpha_1^m \odot \alpha_1^h + \gamma_{1, 1}^k \odot \alpha_1^l \odot \alpha_1^{m+1} \odot \alpha_1^h + \gamma_{1, 1}^k \odot \alpha_1^l \odot \alpha_1^m \odot \alpha_1^{h+1} \mbox{ for } l > m > h, \\
x \cdot (\gamma_{1, 1}^k \odot \alpha_2^l \odot 1_1) = \gamma_{1, 1}^{k+1} \odot \alpha_2^l \odot 1_1 + \gamma_{1, 1}^k \odot (\gamma_{1, 1} \alpha_2^l) \odot 1_1 + \gamma_{1, 1}^k \odot \alpha_1^{l+1} \odot \alpha_1^l \odot 1_1 + \gamma_{1, 1}^k \odot \alpha_2^l \odot \alpha_1, \\
x \cdot (\gamma_{1, 1}^k \odot \alpha_3) = \gamma_{1, 1}^{k+1} \odot \alpha_3 + \gamma_{1, 1}^k \odot (\gamma_{1, 1} \alpha_2) \odot \alpha_1 + \gamma_{1, 1}^k \odot \alpha_1^2 \odot \alpha_2, x \cdot (\gamma_{1, 1} \odot \alpha_1^2 \odot \alpha_2) = \gamma_{1, 1}^2 \odot \alpha_1^2 \odot \alpha_2 \\
+ \gamma_{1, 1} \odot (\gamma_{1, 1}\alpha_2) \odot \alpha_1 + \gamma_{1, 1} \odot \alpha_1^3 \odot \alpha_2, x \cdot ((\gamma_{2, 1}\gamma_{1, 2}^k) \odot \alpha_1^l) = (\gamma_{2, 1}\gamma_{1, 2}^k) \odot \alpha_1^{l+1}.
\end{gather*}

\begin{lemma} \label{restriction-map}
1. Up to degree $ 3 $, the map $ \overline{\rho_3} $ is an isomorphism. \newline
2. In degrees $ 0, 1$, and $2 $ the map $ \overline{\rho_4} $ is surjective. In degree $ 3 $, its cokernel is $ 1 $-dimensional. \newline
3. In degrees up to $ 5 $, the map $ \overline{\rho_5} $ is an isomorphism.
\end{lemma}
\begin{proof}
The Gysin sequence
\[\cdots \rightarrow H^{k-1}(\Bgroup_n) \xrightarrow{\partial} H^k(\Bgroup_n) \xrightarrow{\rho_n} H^k(\Bgrouppos_n) \xrightarrow{\tau_n}H^k(\Bgroup_n) \xrightarrow{\partial}H^{k+1}(\Bgroup_n) \rightarrow \cdots\]
splits into the short exact sequences
\[
0 \rightarrow \frac{H^k(\Bgroup_n)}{(e)} \rightarrow H^*(\Bgrouppos_n) \rightarrow \ker(e \cdot \_ \colon H^k(\Bgroup_n) \to H^{k+1}(\Bgroup_n)) \rightarrow 0.
\]
Therefore, $ \dim(\mathrm{coker}(\overline{\rho}_n)) = \dim (\ker(e \cdot \_)) $.

Let $ n = 3 $, we need to prove that multiplication by $ e = \gamma_{1, 1} \odot 1_1 + \alpha_1 \odot 1_2 $ is injective in degrees $ 0, 1, 2, 3 $. This is true in every degree, but we do not need this more general statement.

Listing all Hopf monomials in the right degrees we obtain the following additive bases:
\begin{gather*}
H^0(\Bgroup_3) = \mathbb{F}_2 \{ 1_3 \}, \quad H^1(\Bgroup_3) = \mathbb{F}_2 \{ \gamma_{1, 1} \odot 1_1, \alpha_1 \odot 1_2 \}, \\
H^2(\Bgroup_3) = \mathbb{F}_2 \{ \gamma_{1, 1}^2 \odot 1_1, \gamma_{1, 1} \odot \alpha_1, \alpha_2 \odot 1_1, \alpha_1^2 \odot 1_2 \}, \\
H^3(\Bgroup_3) = \mathbb{F}_2 \{\gamma_{1, 1}^3 \odot 1_1, \gamma_{1, 1}^2 \odot \alpha_1, \gamma_{1, 1} \odot \alpha_1^2, 1_2 \odot \gamma_1^3, \alpha_3, \alpha_1^2 \odot \alpha_1 \odot 1_1\} .
\end{gather*}
It can be seen by direct computation that the multiplication by the class $ x $ above is injective in these degrees.
\newline
For $ n = 4 $, a similar argument that uses the cup product calculations above shows that multiplication by $ \gamma_{1, 1} \odot 1_2 + \alpha_1 \odot 1_3 $ is injective in degrees $ 0, 1, 2 $, that its kernel in degree $ 3 $ is generated by $ \gamma_{2, 1} $. Similarly, for $ n = 5 $, multiplication by $ \gamma_{1, 1} \odot 1_3 + \alpha_1 \odot 1_4 $ is injective up to degree $ 5 $.
\end{proof}
With the help of the above computations and the lemma, we obtain the following.
\begin{align*}
 H^0(\Bgrouppos_3) &= \mathbb{F}_2, \quad H^1(\Bgrouppos_3) = \mathbb{F}_2 \{ \gamma_{1, 1} \odot 1_1 \}, \\
 H^2(\Bgrouppos_3) &= \mathbb{F}_2 \{ \gamma^2_{1, 1} \odot 1_1= \gamma_{1, 1} \odot \alpha_1, \alpha_2 \odot 1_1 \}, \\
 H^3(\Bgrouppos_3) &= \mathbb{F}_2 \{ \gamma^3_{1, 1} \odot 1_1, \gamma^2_{1, 1} \odot \alpha_1, 1_1 \odot \gamma_{1, 1} \alpha_2 =\alpha_1 \odot \alpha_2 = \alpha_3 \} .
\end{align*}

Likewise, the cohomology groups of $\Bgrouppos_4$ in lower degrees are described below.
\begin{align*}
 H^0(\Bgrouppos_4) &= \mathbb{F}_2 , \quad H^1(\Bgrouppos_4) = \mathbb{F}_2 \{ \gamma_{1, 1} \odot 1_2 \}, \\
 H^2(\Bgrouppos_4) &= \mathbb{F}_2 \{ \gamma^2_{1, 1} \odot 1_2, \gamma_{1, 1} \odot \alpha_1 \odot 1_1, \gamma_{1, 2}, \gamma_{2, 1}, \alpha_2 \odot 1_2 \}, \\
 H^3(\Bgrouppos_4) &= \mathbb{F}_2 \{ \gamma^3_{1, 1} \odot 1_2, \gamma_{1, 1} \odot \alpha^2_1 \odot 1_1=\gamma^2_{1, 1} \odot \alpha_1 \odot 1_1, \gamma_{1, 1} \odot \alpha_2 , \gamma_{1, 1} \alpha_2 \odot 1_2= \alpha_3 \odot 1_1, t \} .
\end{align*}
where $ t $ is a generator of the cokernel of $ \overline{\rho}_4 \colon H^3(\Bgroup_4) \to H^3(\Bgrouppos_4) $. Finally, cohomology groups of $ \Bgrouppos_5 $ up to degree six are given below.
\begin{align*}
 H^0(\Bgrouppos_5) &= \mathbb{F}_2, \quad H^1(\Bgrouppos_5) = \mathbb{F}_2 \{ \gamma_{1, 1} \odot 1_3 \} , \quad H^2(\Bgrouppos_5) = \mathbb{F}_2 \{ \gamma_{1, 1}^2 \odot 1_3, \gamma_{1, 2} \odot 1_1 \}, \\
 H^3(\Bgrouppos_5) &= \mathbb{F}_2 \{ \gamma_{1, 1}^3 \odot 1_3, \gamma_{1, 1}^2 \odot \gamma_{1, 1} \odot 1_1, \gamma_{1, 1} \odot \alpha_2 \odot 1_1, (\gamma_{1, 1}\alpha_2) \odot 1_3, \alpha_3 \odot 1_2 \}, \\
 H^4(\Bgrouppos_5) &= \mathbb{F}_2 \{ \gamma_{1, 1}^4 \odot 1_3, \gamma_{1, 1}^3 \odot \gamma_{1, 1} \odot 1_1, \gamma_{1, 2}^2 \odot 1_1, (\gamma_{1, 1}^2 \alpha_2) \odot 1_1, \gamma_{1, 1}^2 \odot \alpha_2 \odot 1_1, \\
 & \quad \quad \gamma_{1, 1} \odot \alpha_3, \alpha_2^2 \odot 1_3, \alpha_4 \odot 1_1 \} , \\
 H^5(\Bgrouppos_5) &= \mathbb{F}_2 \{ \gamma_{1, 1}^5 \odot 1_3, \gamma_{1, 1}^4 \odot \gamma_{1, 1} \odot 1_1, \gamma_{1, 1}^3 \odot \gamma_{1, 1}^2 \odot 1_1, (\gamma_{1, 1}^3 \alpha_2) \odot 1_3, \gamma_{1, 1}^3 \odot \alpha_2 \odot 1_1, \\
 & \quad \quad \gamma_{1, 1}^2 \odot \alpha_3, (\gamma_{1, 1} \alpha_2^2) \odot 1_3, (\gamma_{1, 1} \alpha_2) \odot \alpha_2 \odot 1_1, \gamma_{1, 1} \odot \alpha_2^2 \odot 1_1, \alpha_2^2 \odot \alpha_1 \odot 1_1, \alpha_5 \} .
\end{align*}

\end{document}